\documentclass{article}
\usepackage[utf8x]{inputenc}
\usepackage[english]{babel}
\usepackage[T1]{fontenc}
\usepackage{lmodern}
\usepackage{fullpage}
\usepackage{graphicx}
\usepackage{epstopdf}
\usepackage{caption}
\usepackage{subcaption}
\usepackage{multirow}
\usepackage{eurosym}
\usepackage{scrextend}
\usepackage{color}
\usepackage{listings}
\usepackage{textcomp}
\usepackage{hyperref}
\usepackage{algorithm}
\usepackage{algorithmicx}
\usepackage[noend]{algpseudocode}

\usepackage{amsmath}
\usepackage{amssymb}
\usepackage{amsthm}
\usepackage{tikz}
\usetikzlibrary{decorations}
\usetikzlibrary{decorations.pathreplacing}

\usepackage[squaren, Gray]{SIunits}
\usepackage{sistyle}

\usepackage{lipsum}
\usepackage[round]{natbib}

\usepackage{booktabs}

\newtheorem{myprop}{Proposition}

\usepackage[a4paper,
            bindingoffset=0.2in,
            left=1.3in,
            right=1.3in,
            top=1.3in,
            bottom=1.3in,
            footskip=.25in]{geometry}

\title{Does Financial Trading Smooth Non-Convex Markets?}

\author{Nicolas Stevens\thanks{Corresponding author.}~\thanks{Max Planck Institute for Research on Collective Goods (Bonn, Germany) \& Center for Operations Research and Econometrics, UCLouvain (Louvain-la-Neuve, Belgium). Email: \texttt{stevens@coll.mpg.de}}~,
Peter Cramton\thanks{Max Planck Institute for Research on Collective Goods (Bonn, Germany) \& University of Maryland, College Park, USA. Email: \texttt{pcramton@gmail.com}}~
and Martial Toniotti\thanks{Center for Operations Research and Econometrics, UCLouvain, Louvain-la-Neuve, Belgium. Email: \texttt{martial.toniotti@uclouvain.be}}}

\begin{document}
\maketitle

\begin{abstract}
In non-convex markets, a competitive equilibrium may fail to exist. This turns out to be an important issue in real-world non-convex auction markets, such as electricity markets, as it complicates pricing and requires the auctioneer to resort to out-of-market discriminatory side payments to sustain an equilibrium.
We investigate whether the introduction of \textit{convex} financial trading induces a \textit{smoothing effect}, mitigating the issues arising from non-convexities.
We develop a two-stage non-convex market model (a forward market followed by a spot market) in which convex financial traders participate in the forward market. Our model predicts that financial trading reduces the magnitude of side payments required to support the cleared allocation.
To test the prediction of our model, we examine the introduction of a transaction fee on financial traders in 2020 by PJM, the US's largest electricity market. We show that the substantial decline in financial trading volume caused by this policy coincided with a significant increase in side payments, in line with our theoretical predictions. 
\end{abstract}


\newpage
\section{Introduction}
In a non-convex market, a competitive equilibrium may fail to exist. 
Several classical results show that this failure might not be too severe, at least \textit{on the large scale}: \cite{starr1969} shows that when an equilibrium does not exist, the distance to an equilibrium (as measured by total excess supply or demand) is invariant with the number of market participants, while \cite{aumann1964} shows that when there is a \textit{continuum} of traders
, an equilibrium exists in spite of the non-convexities.
Yet, despite these ``asymptotic'' positive results, non-convexity remains problematic in real-world auctions where the auctioneer cannot simply overlook these complexities but must determine market-clearing prices \citep{milgrom2017}.
There is an extensive literature which approaches this problem by designing mechanisms that accommodate non-convexities (see, for instance, \cite{milgrom2025}, and the literature covered hereafter). 
Instead of designing new mechanisms, we investigate whether the introduction of financial traders induces a \textit{smoothing effect} which helps solve the issues arising from non-convexities. 
Two important features of financial traders are that (i) they use \textit{convex} bids and (ii) they are \textit{arbitragers}, thus they are expected to bid close to the market margin. 
We analyze how this addition of a large volume of \textit{convex} bids \textit{close to the margin} smooths the non-convexities present in the market and helps the auctioneer with price formation.
We study this question in the context of electricity markets, which combine non-convex auction structures with active financial trading.

Electricity wholesale markets are typically organized as sealed-bid uniform-price auctions \citep{wilson2002,cramton2017electricity}. Most of these auctions involve non-convexities in the bidding language that enable the market participants, in particular the electricity producers, to represent the costs and constraints of operating a power plant directly into the auction.
Focusing on the United States\footnote{Although our work is focused on the US markets, non-convexities are also present in other electricity markets. For instance, non-convexities arise from the ``block orders'' (with features such as all-or-nothing dispatch, minimum acceptance ratio, exclusive groups, etc.) in European or Indian markets.}, each power plant bids separately in the market, using multi-parts bids that represent their technical characteristics, and where non-convexities arise from features such as start-up costs, minimum up and down times requirements, minimum production limits, and so on. 
The Walrasian auctioneer collects the bids of all the market participants and solves the so-called unit commitment model \citep{knueven2020} in order to find the allocation that maximizes the market surplus under all technical constraints (grid constraints and production constraints represented by the bids of the market participants). The auctioneer  then computes the market-clearing prices.

The presence of non-convexities has several consequences, such as on the algorithmic complexity of the auction \citep{knueven2020}, or on the bidding behaviour and the nature of competition in the market \citep{reguant2014,jha2025}. Yet, the main implication is that an equilibrium is not guaranteed to exist, although it might exist in some cases \citep{bikhchandani1997}. 
In concrete terms, given the surplus-maximizing allocation, the auctioneer may not be able to find a uniform price that ``clears the market''.
This is illustrated in the following example.
Figure \ref{fig:exampleSidePayments} presents an elementary electricity market with two demand bids, $D_1$ and $D_2$, as well as two supply offers $S_1$ and $S_2$. Assuming first that all the bids are convex, a Walrasian auctioneer maximizing the total surplus would simply intersect supply and demand: he would clear $D_1$, clear 90MWh of $S_1$ and reject the other bids. In this configuration, 
the price $\pi = 30$\$/MWh clears the market.
If now $S_1$ is a non-convex ``all-or-nothing'' bid (either it is entirely cleared or entirely rejected), then the surplus-maximizing allocation clears $S_1$, $D_1$, as well as 10MWh of $D_2$. 
The key issue arising from this non-convexity is that there is no uniform price that supports the cleared allocation. 
Indeed, if the auctioneer chooses $\pi = 20$\$/MWh, then $S_1$ loses money. More generally, this occurs whenever $\pi < 30$\$/MWh. 
If instead he chooses $\pi = 30$\$/MWh, the demand bid $D_2$ loses money since his willingness-to-pay for electricity is below the market price. This occurs whenever $\pi > 20$\$/MWh. Thus there is no competitive equilibrium (no price satisfies both \textit{market-clearing} and \textit{envy-freeness}).

\begin{figure}
\centering
\includegraphics[width=0.5\textwidth]{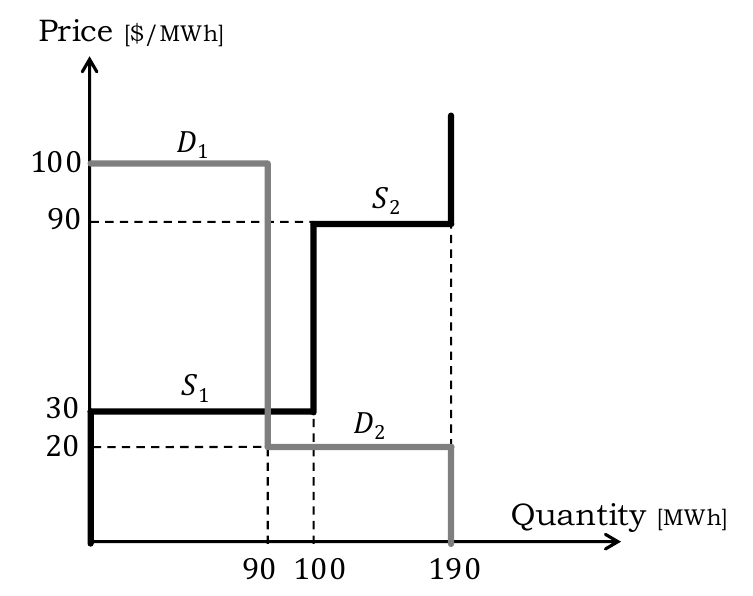}
\caption{An hourly two-sided electricity market with four bids.}
\label{fig:exampleSidePayments}
\end{figure}

Practically, the electricity auctioneers address the issue by complementing the uniform price with discriminatory \textit{side payments} (or ``markups'', as called by \cite{milgrom2025}) paid to the market participants to support the allocation. In our example, this could mean setting the price to $\pi = 20$\$/MWh while providing a side payment of $(30 - 20) \times 100 = 1000$\$ to $S_1$ in order for $S_1$ to break even. 
It can be shown more formally, and for a general auction model, that a competitive equilibrium can indeed be recovered by extending the uniform price vector with side payments \citep{oneill2005}.
These side payments are generally regarded as undesirable for they are discriminatory, non transparent (they do not have temporal nor locational granularity), they can induce gaming behaviour \citep{byers2023auction} and they also deter investment signal \citep{byers2023}.
The key implication is that \textit{some relevant cost information is not reflected in the uniform price signal but is instead handled by ``out-of-market'' side payments}.

There is an extensive literature that aims at addressing these issues by designing pricing mechanisms for non-convex auctions \citep{milgrom2025}, and in particular for electricity auctions \citep{oneill2005,gribikHogan2007,chao2019,stevens2024,ahunbay2025,stevens2026}. 
These market design proposals depart from the classical marginal pricing rule by incorporating some non-convex costs into the price signal, thereby reducing the side payments required to sustain equilibrium. For example, one prominent proposal is to select the uniform price that \textit{minimizes} the required side payments \citep{gribikHogan2007}.

In this paper, we approach the problem from a different perspective by asking whether the participation of financial traders can itself help mitigate these issues.
Electricity auctions are not pure combinatorial auctions, as they include a significant share of \textit{convex} bids, among which are the financial trading bids, often referred to as ``virtual bids''.
Electricity markets are organized as sequential auctions, the most important being the day-ahead \textit{forward} market and the real-time \textit{spot} market.
The day-ahead auction is a financial market that involves both ``physical'' (non-convex) bidders, tied to physical assets like power plants, and ``virtual'' (convex) bidders that engage in arbitrage between day-ahead and real-time prices.
The benefits of financial trading in electricity markets have been extensively studied in the literature \citep{hogan2016virtual}. These benefits include mitigating market power \citep{ito2016,mercadal2022}, enhancing price convergence and reducing costs \citep{jha2023}, as well as providing risk hedging opportunities \citep{longstaff2004}.
Hogan further conjectured the additional benefit that ``including virtual transactions [...] would have the effect of smoothing the day-ahead commitment and dispatch problem and reducing the required uplift [side] payments.'' \citep{hogan2016virtual}
In this paper, we investigate this claim formally and empirically: by making the non-convex market more convex, financial trading exerts a \textit{smoothing effect} that improves price formation and reduces the side payments required to sustain an equilibrium.
In particular, our paper makes two main contributions. 

First, we develop a two-stage stochastic market model (a forward market followed by a spot market) in which financial traders participate in the forward market. Our model predicts that financial trading improves price convergence, reduces total cost, and---most important for our analysis---mitigates the magnitude of side payments required from the auctioneer to sustain an equilibrium. 
More specifically, we show that virtual trading induces two distinct effects.
(i) First, by behaving as arbitragers in the day-ahead market, virtual traders internalize both expected real-time conditions and the non-convex operating costs embedded in day-ahead commitment decisions.
In our model, virtual traders set the price and eliminate the side payments in day-ahead.
(ii) Second, by anticipating real-time conditions, financial trading implies better commitment decisions in day-ahead, leading to less lumpy activations of ``fast-start resources'' in real-time. This helps reduce the side payments, as the activation of fast-start resources is typically associated with side payments needed to cover their startup costs. 

This \textit{smoothing effect} is distinct from the classic \textit{market size effect} as analyzed by \cite{starr1969} in the context of general equilibrium theory and recently reinterpreted by \cite{milgrom2025} in the context of auction theory (see also \cite{chao2019}, \cite{stevens2024} and \cite{hubner2025} for transposition of these results to the context of electricity auctions).
The market size effect shows that, in a non-convex economy where equilibrium may fail to exist, the distance to equilibrium---as measured by either total excess demand or side payments---is bounded by a constant independent of the number of market participants. Hence, as the market becomes larger (increasing the number of participants), the absolute distance to equilibrium remains unchanged while its relative importance with respect to market size decreases.\footnote{This important result is a direct application of Shapley-Folkman lemma. See also \citep{bertsekas1982,bertsekas1983} for the perspective of linear programming and duality gaps.}
The smoothing effect we analyze is not merely related to the market size, but to the specific type of \textit{convex arbitrage bids} introduced by the financial traders.

Second, we test the prediction of our model by analyzing the introduction in 2020 of a transaction fee on financial traders in PJM, the largest electricity market in the United States. This transaction fee led to a substantial decline in financial trading volume. We provide evidence that this reduction was followed by a significant increase of side payments needed to sustain equilibrium.
We find that the real-time side payments increased by 80\% following the introduction of the transaction fee, while the likelihood of non-zero day-ahead side payments increased by 10\%.

Our empirical analysis contributes to the literature on financial trading, with findings in line with prior empirical evidence.
In particular, \cite{long2020} simulate the PJM auction model with and without financial bids and show that the inclusion of financial bids mitigates the side payments. The main difference with our analysis is that their work relies on simulations that assume the removal of financial bids does not affect the bidding behavior more broadly. While our empirical analysis relies on a policy change in the market. 
Besides these simulation exercises, the two closest empirical works are \cite{hubner2025} and \cite{jha2023}.
\cite{hubner2025} 
provides empirical evidence, in the context of the European electricity auction, that when the market is large, with a high share of convex bids, the probability of reaching an equilibrium becomes higher.
The main difference with our work, besides the fact that we focus on PJM market instead of the European market\footnote{There are two main benefits of focusing on PJM market over the European market, for our analysis: (i) non-convexities are more prevalent in PJM market (as PJM market is nodal, the market tends to be more fragmented, which exacerbates the impact of non-convexities ; and the European market bidding language, although it includes non-convex bids, tends to be somewhat less non-convex than the unit commitment model of PJM) and (ii) financial traders' bids can be distinguished from power plants' bids.}, is that the analysis of \cite{hubner2025} is focused on the \textit{market-size effect}---considering the effect of convex bids in general---while our work focuses on the \textit{smoothing effect} induced by a specific type of convex bids, namely the financial trading bids.
Finally, \cite{jha2023} provide empirical evidence of price convergence and cost-saving entailed by the participation of financial traders in CAISO market. While they consider a non-convex market, they do not specifically look at the implications for side payments and price formation.

The rest of the paper is organized as follows. Section \ref{sec:PJMmarket} provides an overview of PJM Interconnection electricity market. In particular, it discusses the sequential nature of the market (forward and spot markets) and the role of financial traders, and it provides descriptive statistics of the share of convex and non-convex bids in the market as well as the magnitude of the side payments.
It also helps motivating some modeling choices we later make.
Section \ref{sec:model} goes on with our model of a non-convex sequential electricity market with financial trading participation. This section formalizes the intuition of the smoothing effect described above.
Sections \ref{sec:empiricalStrat} and \ref{sec:mainResults} cover the empirical analysis with the data from PJM.
Section \ref{sec:ccl} puts our main findings in perspective of policy discussions and concludes.

\section{PJM Interconnection market} \label{sec:PJMmarket}
PJM Interconnection operates the largest wholesale electricity market in the United States, serving 67 million people across 13 states, and monitoring more than 1,400 power plants accounting for roughly 180GW of generation capacity.\footnote{The installed capacities in 2024 were roughly 90GW of gas, 35GW of coal, 34GW of nuclear, 5GW of oil, 6GW of hydropower, >10GW of wind and >10GW of solar.}
In 2024, the hourly average demand including exports was 95 GWh (with a peak load plus exports of 154 GWh), served by production coming mainly from gas (44\%), coal (15\%) and nuclear (32\%) \citep{MonitoringAnalytics2025}.

As it is common in the US, the wholesale electricity market works as a sequence of sealed-bids uniform price auctions organised by PJM, which follow the principles of bid-based, security-constrained, economic dispatch with locational marginal prices \citep{hogan2021strengths}. Two cornerstones of the wholesale market are the real-time \textit{spot} market and the day-ahead \textit{forward} market.
The real-time market is the \textit{physical} market when a physical commodity is sold ``on the spot'' and where the trades correspond to physical exchanges of energy.
The day-ahead market is a \textit{financial} market: there is no physical delivery but only trades of financial contracts, which are derivatives of the real-time price of electricity. Yet, it is a financial market which has significant physical implications as its outcome is used for coordination and commitment of power plants for the next day \citep{stoft2002}.
Generation companies and consumers' electricity providers participate in these auctions by submitting multi-parts bids representing their willingness to sell or to buy electricity.
PJM market also includes a detailed model of the power grid representing the constraints of trade between the locations (thousands of nodes are represented in PJM auction model). 
The real-time market clears every five minutes while the day-ahead market clears once per day (it is a multi-periods market that clears at once the 24 hourly periods of the next day). These markets also co-optimize both energy and ancillary services products.\footnote{See PJM manual for a detailed timeline and specific rules of these market \citep{PJM2025b}.}

\paragraph*{Bids and non-convexities.}
The day-ahead electricity auction in PJM mainly includes two types of bids:  the ``physical'' bids and the ``virtual'' bids.
Physical bids are tied to physical assets. They further split into two types: dispatchable generation bids and self-schedule bids. The dispatchable generation  are \textit{non-convex} (or partly non-convex) bids that enable the market participants to express the true economics of their power plants: they are multi-part ``unit commitment'' bids which include explicit information about costs and constraints of production. The main sources of non-convexities relate to production constraints such as minimum production limits or minimum up and down times requirements of power plants as well as to fixed costs such as start-up costs and no-load costs (let us emphasize that whenever we talk about ``fixed costs'' in this paper, we mean the operational avoidable fixed costs, such as start-up costs, as opposed to the capital costs which are sunk and have no implication for our discussion).
The self-schedule bids are convex and correspond to ``offer to supply a fixed block of MW, as a price taker'' \citep{MonitoringAnalytics2018}, which in PJM means a continuous quantity at zero price.
Power plants have the obligation to participate into the day-ahead market of PJM, but they are allowed to either participate as dispatchable generation or to self-schedule \citep{PJM2025b}.\footnote{The main restriction to self-scheduling is that these bids are not compensated for their losses in the market: to incentivize units to be flexible, self-schedule units are not eligible to receive side payments.}
Figure \ref{fig:PJMbids_meritOrder} shows the merit order curve of physical bids for a sample hour (12/06/2019 at 9am): the day-ahead self-schedule bids are the flat zero-price segment while the dispatchable generation bids are the remaining part of the curve. 
This aggregated merit order curve is of course a simplification, as it ignores the network constraints (although the selected period corresponds to an hour with little congestion), the fixed costs and other constraints of production. Yet, we see on Figure \ref{fig:PJMbids_meritOrder} that supply and demand (the inelastic load) cross closely to the actual PJM system price at that time.
The figure also shows the \textit{real-time} self-scheduling, which corresponds to the amount of inflexible supply in real-time. As we observe, this coincides to most of the supply committed in day-ahead.

\begin{figure}
    \centering
    \begin{subfigure}[b]{0.8\textwidth}
        \centering
        \includegraphics[width=\textwidth]{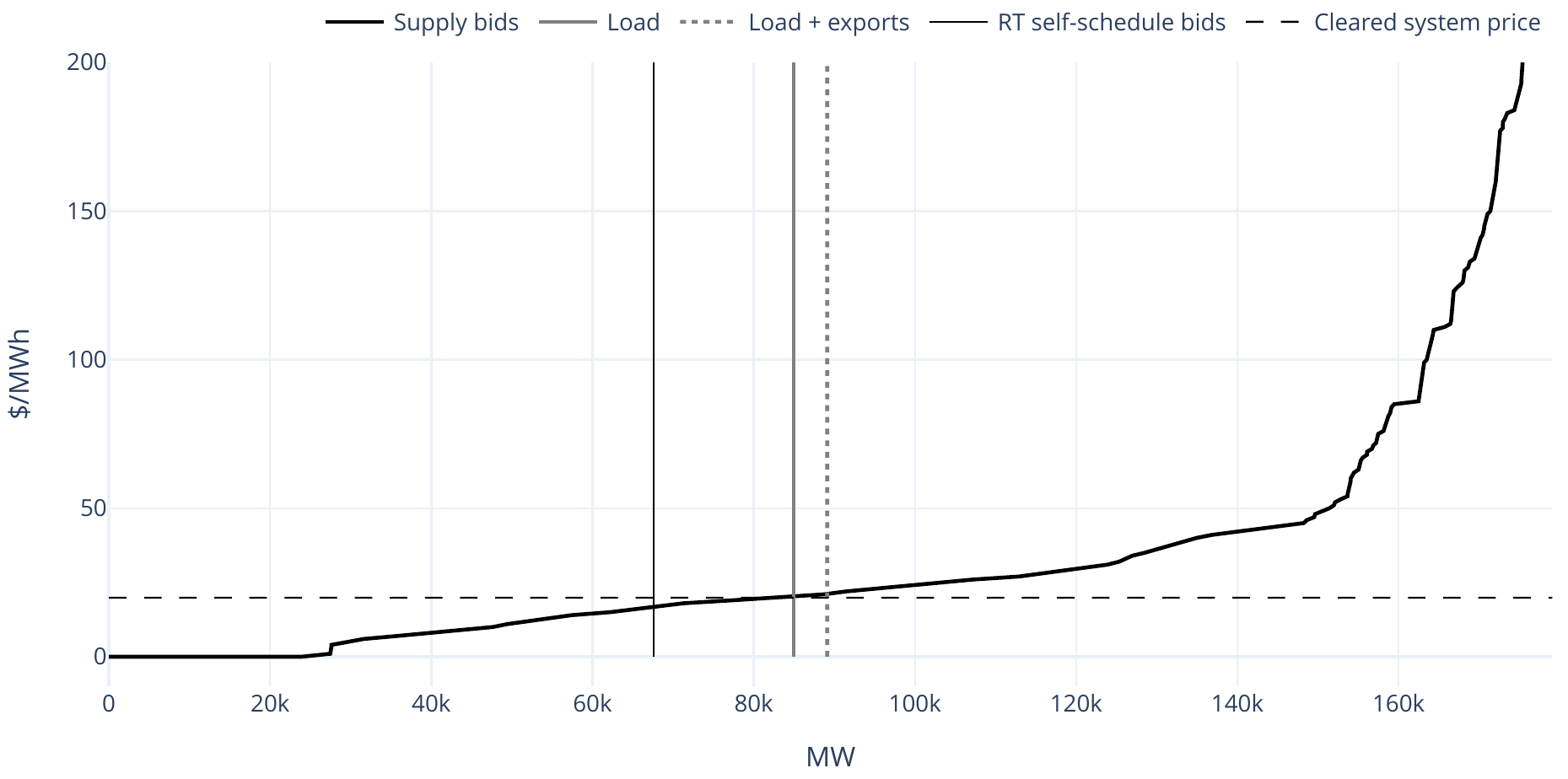}
        \caption{PJM merit order curve (physical bids)}
        \label{fig:PJMbids_meritOrder}
    \end{subfigure}
    \begin{subfigure}[b]{0.8\textwidth}
        \centering
        \includegraphics[width=\textwidth]{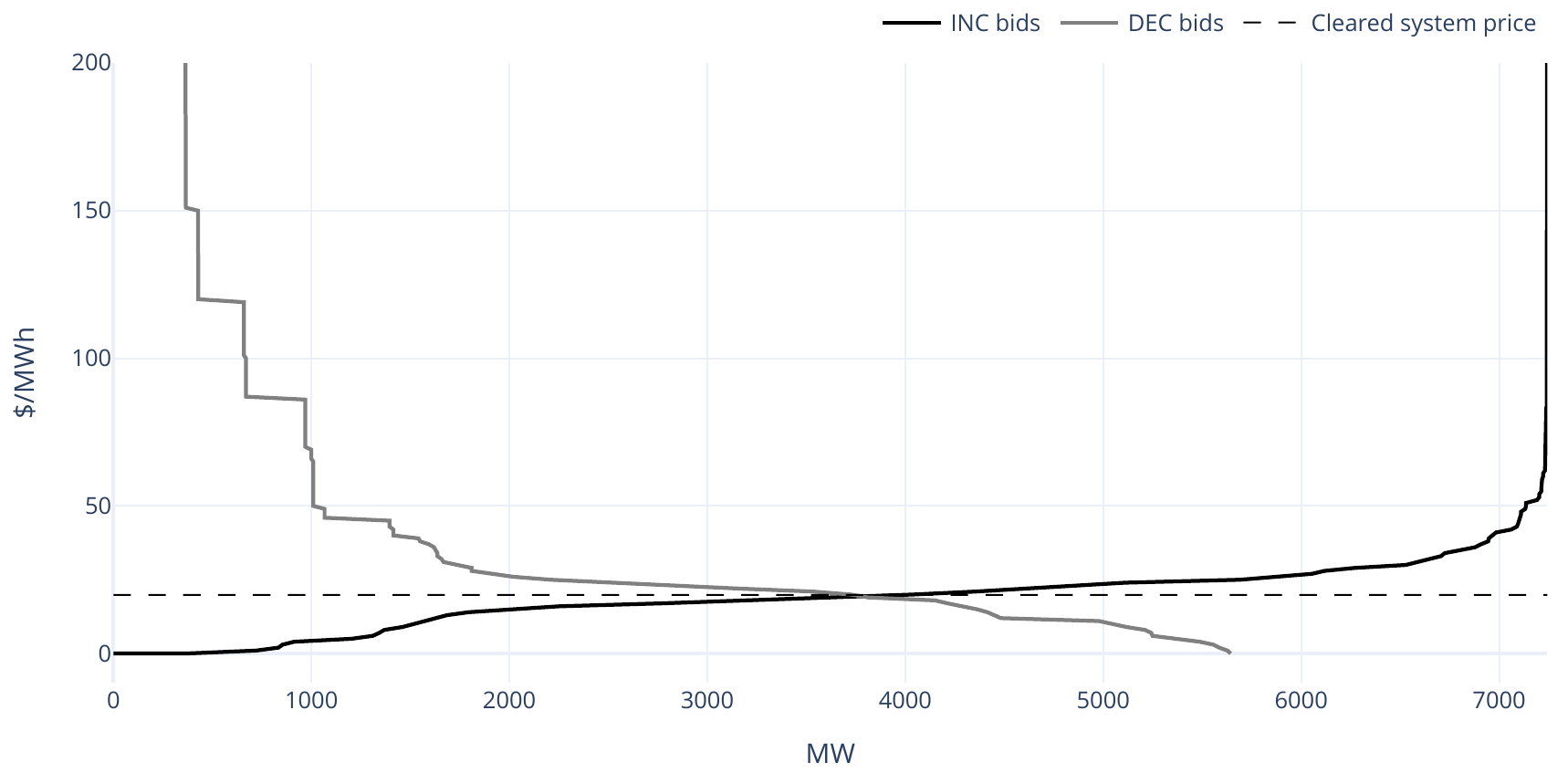}
        \caption{PJM virtual bids (INC and DEC)}
        \label{fig:PJMbids_VB}
    \end{subfigure}
    \caption{PJM bidding curves by category on June 12th, 2019, at 9am.}
    \label{fig:PJMbids_curves}
\end{figure}

Besides these physical bids, PJM day-ahead market also includes virtual bids which are pure financial trading: market participants playing as arbitrager between the day-ahead financial market and the real-time spot market.
Concretely, a virtual trader who sells a quantity $q$ at the day-ahead price $p_{DA}$ have to buy it back at the real-time price $p_{RT}$, thus he makes a profit $q \times (p_{DA} - p_{RT})$.
There are three types of virtual bids in PJM: INC, DEC and UTC orders, which are all \textit{convex} bids. An INC order is a sell, or ``increment'', order (an injection) in one location of the grid. A DEC is a buy, or ``decrement'', order (a withdrawal) in one location of the grid. A UTC is Up To Congestion Transactions order which is a combination of an INC and a DEC order in two different locations: an arbitrage between two locations of the grid, similar to a FTR contract \citep{hogan2016virtual}.
Figure \ref{fig:PJMbids_VB} also shows the INC and DEC bid curves on the same sample hour as Figure \ref{fig:PJMbids_meritOrder}. We observe these curves intersect closely around the market price and are fairly flat around the market price, thus \textit{these virtual bids add a large volume of convex bids close to the margin}.
That is: unlike self-schedule bids, which are convex \textit{but at zero price}---thus not at the margin---, virtual bids are convex \textit{and close to the margin}, which means more likely to set the price.
This is a crucial observation for the smoothing effect analysed in our paper, which will be formalized later in section \ref{sec:model}.
Note that there is an explicit distinction in PJM market data between physical and financial bids. This is a major asset for our analysis, and it contrasts with European electricity markets, where bids often aggregate multiple assets and do not separate physical from financial offers. There, a high share of convex bids does not necessarily reflect active financial trading, but might relates to a more convex physics: a smoother operating day with fewer binding non-convex constraints. PJM data allows us to cleanly disentangle these two effects.

From the description of these different bids, it emerges that electricity auctions such as PJM are not pure combinatorial auctions: although they include non-convexities, they also include a fair amount of convex bids. 
Figure \ref{fig:PJMbids_categories} shows the daily \textit{cleared} volumes of INC, DEC and UTC as well as day-ahead self-schedule bids as a percentage of the total day-ahead load. 
Table \ref{tab:BidVolumePerCat} also provides more figures of the hourly average volume of the different category of bids, focusing on the period 2018 to 2021 which will be used in our empirical analysis.\footnote{We report here the \textit{cleared} volume of virtual bids (and not the bided volume). Self-scheduling are bids at a price of zero, thus bided and cleared volume should be the same. The non-convex dispatchable bids are computed using the data of dispatchable bids and keeping only the ones with positive fixed costs.}
We observe that the volume these convex bids is a significant share of the market.

\begin{figure}
\centering
\includegraphics[width=0.9\textwidth]{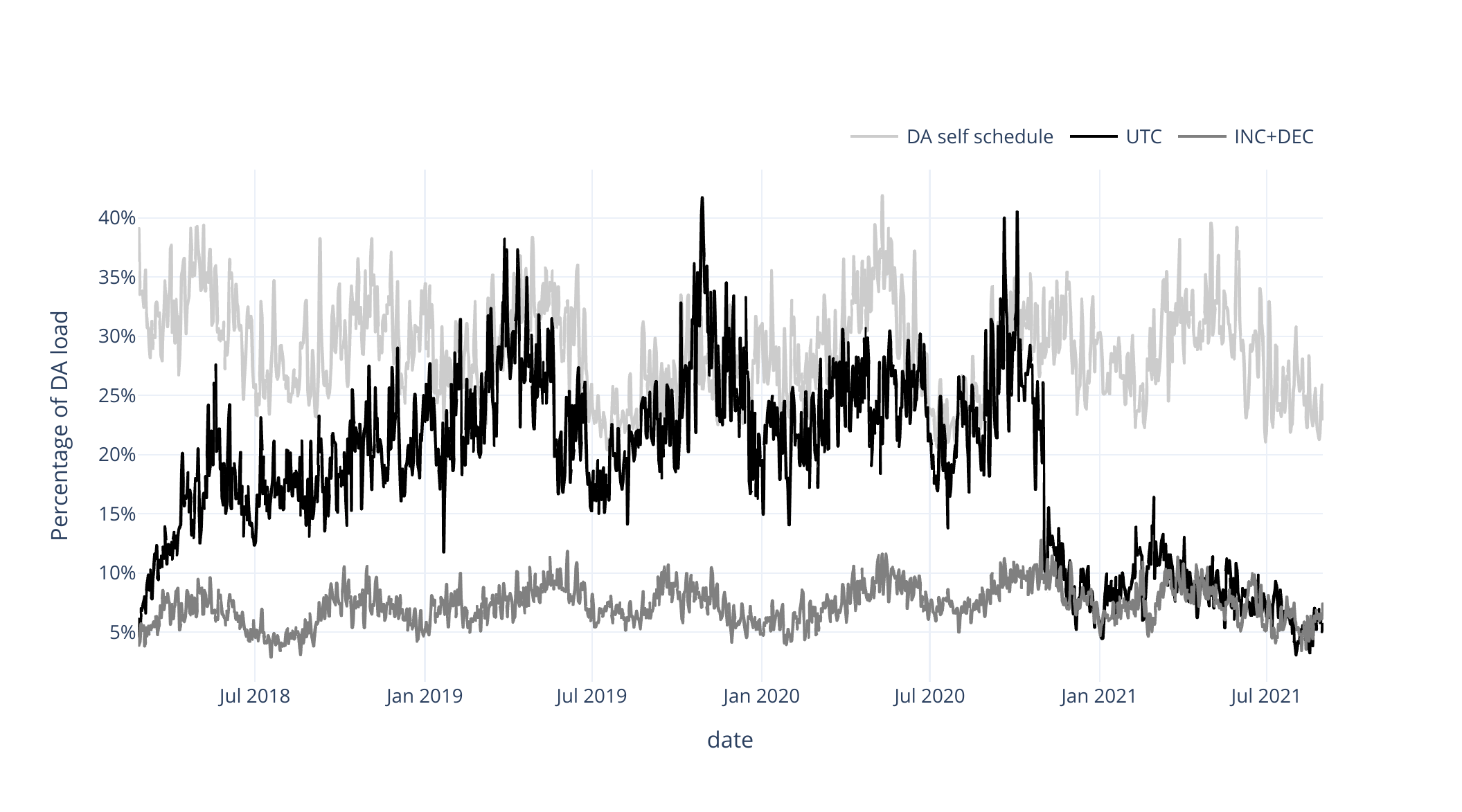}
\caption{PJM convex bid volume as a percentage of total load.}
\label{fig:PJMbids_categories}
\end{figure}

\begin{table}
\centering
\begin{tabular}{ll}
\toprule
Average hourly volumes & GW \\
\hline
Day-ahead load & 88 \\
\textit{Physical bids:} & \\
-- Available non-convex dispatchable bids & 139  \\
-- Day-ahead self-schedule bids  & 25  \\
\textit{Virtual bids:} & \\
-- Cleared INC bids & 2.5  \\
-- Cleared DEC bids & 3.8  \\
-- Cleared UTC bids & 15.4  \\
\bottomrule
\end{tabular}
\caption{Average hourly volume of bids per category, over the period 2018 to 2021 in PJM.}
\label{tab:BidVolumePerCat}
\end{table}

\paragraph*{Side payments.} 
Because an equilibrium is not guaranteed to exist in a non-convex market, PJM relies on discriminatory side payments (sometimes called ``uplift payments''\footnote{``Uplifts'' payments are paid for various reasons which are not all related to the non-convexities \textit{per se}, such as ancillary services uplifts and reliability services uplifts (cf. PJM manual 28 on Operating Agreement Accounting for rules on how these uplifts are computed \citep{PJM2025}). 
The side payments directly related to the non-convexities, which are shown in Figure \ref{fig:PJMuplifts} are the ``operating reserve uplifts'' paid both in the day-ahead and in the real-time markets, which correspond to: ``The total resource offer amount for generation, including startup and no-load costs as applicable, is compared to its total energy market value for specified operating period segments during the day [...] If the total value is less than the offer amount, the difference is credited to the PJM Member'' \citep{PJM2025}
A units whose day-ahead revenue, at the day-ahead price, does not cover his cost receives a DA operating reserve uplift credit. Then, in real-time, this unit might have a schedule that deviates from its day-ahead plan. She then receives balancing (real-time) operating reserve uplift credits for loss exposure in the spot market due to following PJM dispatch. 
We discuss more the different types of uplifts in Appendix \ref{sec:appendixPJMdata}.}) in order to restore the incentives of market participants to follow the market allocation. 
We use these side payments as a measure of how far the price-allocation pair stands from an equilibrium.
It can be shown more formally that when an equilibrium does not exist, the distance to an equilibrium is measured by the side payments (more formally, by the ``lost opportunity costs''\footnote{\cite{starr1969} and \cite{arrow1971} use total excess supply or demand to measure the distance to equilibrium: given a price, each participant maximizes her profit but the market clearing constraint might be violated as an equilibrium does not exist. This market clearing constraint violation (excess supply or demand) is used as a measure of distance to equilibrium.
Another way to measure the distance to equilibrium, which is natural in context of auctions where the auctioneer determine both the price and the cleared quantities, is with the concept of \textit{lost opportunity costs} (cf. \cite{stevens2024} for a formal argument and a detailed discussion). By definition, the quantities cleared by the auctioneer satisfy the market clearing constraint, but the cleared quantities might not maximize the profit of the market participants given the cleared price: they violate envy-freeness. The lost opportunity costs measure the monetary incentives that market participants have to deviate from the cleared quantities.
In particular, some market participants might be loosing money given the cleared price and quantities (cf. the example of Figure \ref{fig:exampleSidePayments}).
Therefore, the auctioneer complements the uniform price with side payments to make sure that exposed power plants break even.
These side payments typically only stand for a share of the lost opportunity costs, as some ``lost opportunities'' are not compensated (e.g. a market participant that is not producing, but could have made a positive profit at the cleared price is not compensated for his lost opportunity).
Thus the side payments is only approximative measure of the distance to equilibrium.}, cf. \cite{stevens2024}).
Small side payments are an indication that the cleared quantities and the uniform price are close to be an equilibrium, while large side payments are an indication the market outcome is further away from an equilibrium.
Large side payments indicate that the price signal does not reflect some relevant cost information.
Figures \ref{fig:PJMuplifts} shows the monthly side payments related to the non-convexities.
We observe that they stand for 7.9M\$/month and that most of these side payments are paid in real-time rather than in day-ahead.

\begin{figure}
\centering
\includegraphics[width=0.9\textwidth]{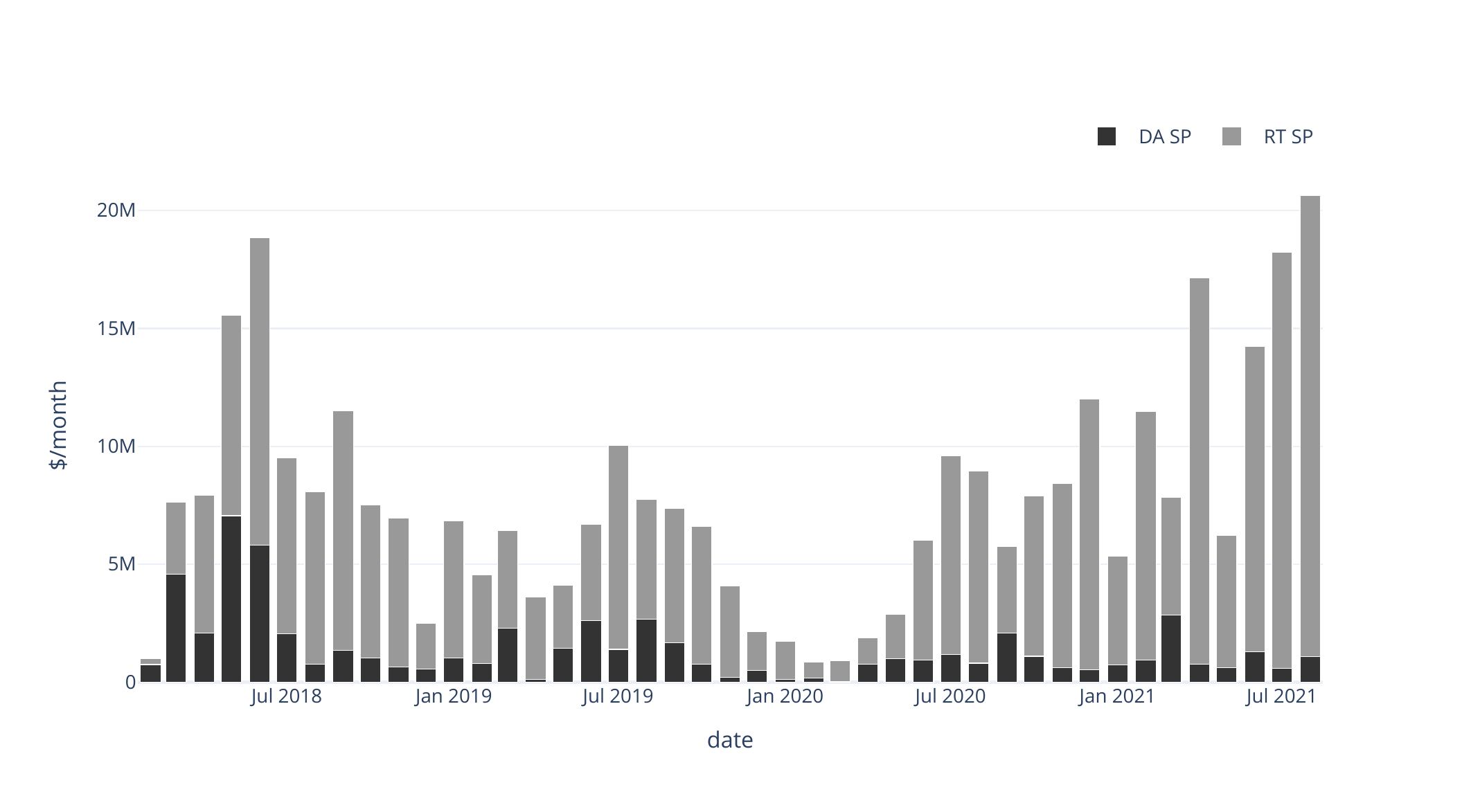}
\caption{PJM side payments due to non-convexities in day-ahead and real-time.}
\label{fig:PJMuplifts}
\end{figure}

\section{Model} \label{sec:model}
\paragraph*{Model environment.}
Let us consider a sequence of two markets: the day-ahead (DA) forward market and the real-time (RT) spot market. 
This sequence of markets are represented on Figure \ref{fig:modelEnv1}.
There are two types of agents: the virtual bidders, active in the DA market, who play as arbitrager ; and the physical bidders, active in both the DA and RT markets, who are producers of electricity, tied to some generation assets. 
The physical bidders incur both a fixed start-up cost $s$ and a variable cost of production.
These physical producers further split into two categories: the fast-star resources (FSR) and the slow-start resources (SSR). The FSR are flexible assets, which can be started-up in real-time even though they are not committed in the DA market. On the contrary, the SSR that are not committed in DA cannot participate in the RT market (this assumption aligns with Figure \ref{fig:PJMbids_meritOrder}, where we saw that a large share of the units committed in day-ahead self-schedule in real-time, thus are inflexible).
An SSR unit who is committed in DA can still be turned off in RT but the start-up cost $s$ is sunk.
There is an exogenous inelastic demand which is made of a predictable load $d$ and a random shock $\epsilon \in [-\overline{D}, \overline{D}]$ on demand, which is uniformly distributed: in DA, the demand is $d$ ($> \overline{D}$), while in RT, the demand is $d+ \epsilon$.
This exogenous shock should be viewed as representing all sorts of uncertainties that arise between day-ahead and real-time and which affect the net load served by SSR and FSR: uncertainties on the load, on wind production, on the availability of some base-load units (for instance due to forced outages), or on imports, exports and grid constraints.

\begin{figure}
\centering
\includegraphics[width=0.6\textwidth]{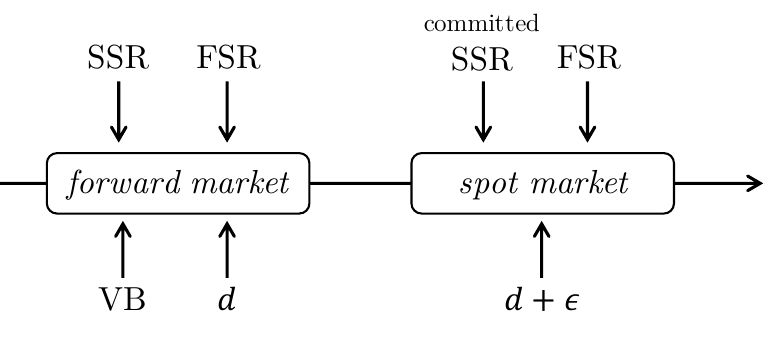}
\caption{Model environment: market timing.}
\label{fig:modelEnv1}
\end{figure}

We assume physical bidders bid truthfully in the sequential markets and that they are myopic: they do not anticipate the real-time conditions when bidding in day-ahead. 
To qualify this seemingly restrictive assumption, let us stress that physical assets could also participate as virtual bidders. Thus the myopic assumption should be viewed as an assumption that the arbitrage is made entirely through the virtual bids. 
In the model, we consider a continuum of slow start-units with a total cost equal to $C_S(q) = \frac{c_S}{2} q^2 + s_S q$ and a continuum of fast start-units with a total cost equal to $C_F(q) = c_F q + s_F q$.
There is a base-load slow-start unit of capacity $d$ with marginal cost equal to 0 and a fixed cost of $s_S d$.
We assume $c_F > c_S \overline{D} + s_S$ and $s_F > s_S$. That is, fast-start units are always more expensive than slow-start units.
These settings are represented on Figure \ref{fig:modelEnv}.

Note that our modelling assumptions (a two-stage market with perfectly inelastic demand, random real-time shocks, and higher-cost resources in real time than in the day-ahead market) provide a reasonable approximation of PJM market operations (cf. section \ref{sec:PJMmarket}).
They are also consistent with standard frameworks of financial trading encountered in the literature \citep{tang2016,mather2017}. In particular, the distinction between slow- and fast-start resources is well established in the electricity market literature \citep{eldridge2023,kazempour2017a,kazempour2017b}.

\begin{figure}
\centering
\includegraphics[width=0.95\textwidth]{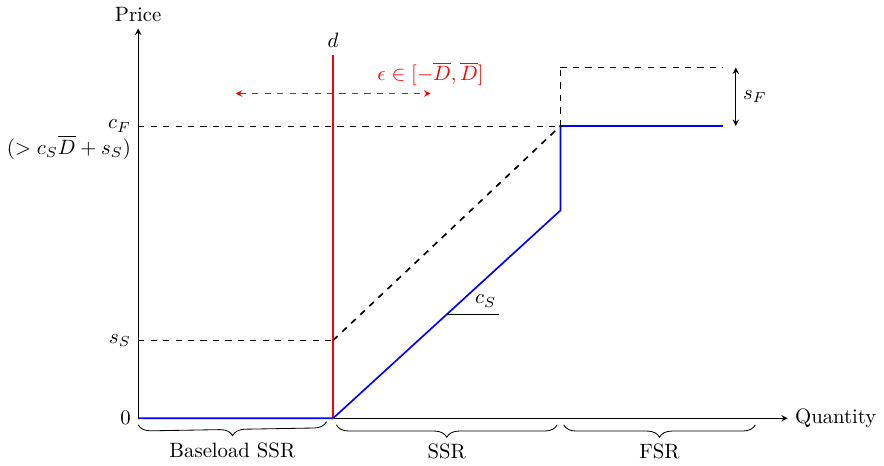}
\caption{Model environment: market participants.}
\label{fig:modelEnv}
\end{figure}

\paragraph*{Mechanism description.} Our model of a \textit{non-convex} market with arbitragers is slightly artificial, for it is in fact \textit{convex}. 
This modeling approach is justified by the fact that including non-convexities renders derivation of analytical results intractable. Yet, by making some assumptions on the allocation rule and settlement rule of the mechanism, we can mimic the behaviour of a non-convex auction.
We shall solve numerically at the end of this section an instance of a truly non-convex market to illustrate and confirm our analytical findings.
Meanwhile, we assume the following mechanism, illustrated in Figure \ref{fig:modelEnvMechanism}:
\begin{itemize}
\item \textit{Allocation rule}: the auctioneer clears the allocation that maximizes social welfare. That is: demand is cleared as long as its willingness-to-pay is higher than average cost of the next infinitesimal producers.
Figure \ref{fig:modelEnvMechanism} shows the two main scenarios. In scenario $VB_1$, the demand curve from the virtual bids intersects horizontally with the average cost curve of the SSR. These bids are thus partially cleared by the auctioneer until they cross with the average cost of SSR. Indeed, clearing $VB_1$ entirely would lower social welfare as their willingness-to-pay is lower than the average cost of the SSR above their curve. In scenario $VB_2$, the demand curve intersects vertically with the average cost curve of the SSR. These bids are fully cleared by the auctioneer.
\item \textit{Pricing and settlement rule}: the auctioneer uses marginal pricing rule, thus the marginal agent sets the price. In scenario $VB_1$, the virtual bids are partially cleared and thus set the price: $p_1 = VB_1$. In scenario $VB_2$, the virtual bids are fully cleared and the price is set by the marginal SSR: $p_2 = c_S VB_2$.
Producers who do not break even receive side payments to offset their losses.
\end{itemize}
An important observation is that the fixed costs $s_F$ and $s_S$ enter the allocation rule's computations, whereas they are not necessarily accounted for in the pricing rule.
This mimics the typical two-stage procedures used by electricity auctioneers who first solve a comprehensive unit commitment model to compute the allocation and then solve a convex restriction of the original problem in which the commitment decisions are fixed, and thus the startup costs omitted from the pricing problem \citep{oneill2005}.

\begin{figure}
\centering
\includegraphics[width=0.95\textwidth]{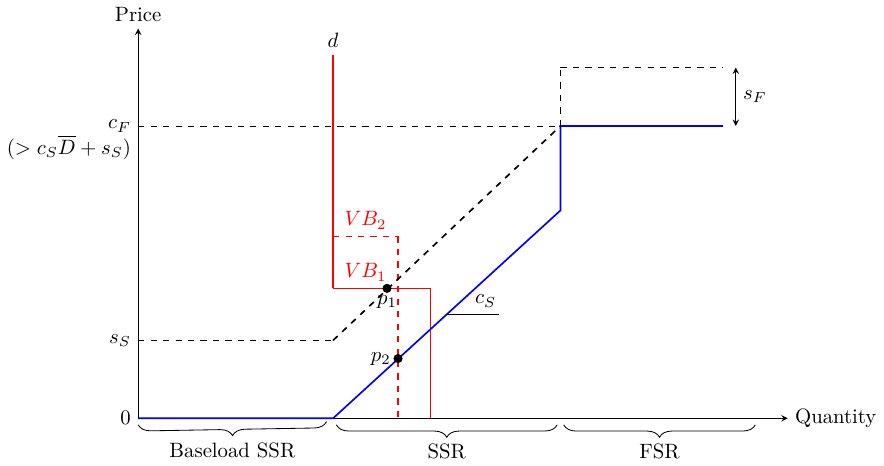}
\caption{Model environment: allocation and pricing rules.}
\label{fig:modelEnvMechanism}
\end{figure}

\paragraph*{Equilibrium without virtual trading.}
Without virtual trading, the equilibrium prices and side payments in day-ahead and real-time are:
\begin{align}
& p_{DA} = 0, ~~~~ SP_{DA} = s_S d \\
& \mathbb{E}_{\epsilon} (p_{RT, \epsilon}) = \frac{c_F}{2}, ~~~~ \mathbb{E}_{\epsilon} (SP_{RT,\epsilon}) = \frac{s_F \overline{D}}{4}
\end{align}
In day-ahead, the lumpy baseload unit covers the entire production, and there is no other activation of slow-start units. This is followed in real-time by lumpy activations of fast-start units, depending on the shock $\epsilon$. 
The baseload unit sets the price to zero in day-ahead, as well as in real-time when there is a negative shock. When the shock is positive, the fast-start units set the price in real-time. There is a price spread between DA and RT markets---as we shall see, virtual traders will arbitrage away this spread.
As the price does not cover the fixed costs, side payments are paid in DA to cover the fixed costs of the base-load SSR. Side payments are also paid in RT to cover the fixed costs related to the activation of FSR.
The total expected cost corresponds to both the baseload unit cost and the expected cost of activating fast-start units:
\begin{equation} \label{eq:costNoVB}
Cost = s_S d + \frac{\overline{D}}{4}(c_F + s_F)
\end{equation}

\paragraph*{Equilibrium with virtual trading.}
Let us consider the case of ``perfect'' virtual trading, where virtual traders compete \textit{\`a la Bertrand}: they compete by outbidding each other until the no-arbitrage condition holds.
This corresponds to arbitragers who bid an infinite quantity at a price $P$ such that $p_{DA} = \mathbb{E}_{\epsilon} (p_{RT, \epsilon})$. 
As shown in Figure \ref{fig:PJMbids_VB}, although actual virtual traders in PJM do not submit bids at the exact same price $P$, their bids remain relatively flat around the market price.

Virtual bidders anticipate the risk of real-time activations of fast-start units and their impact on prices. As a consequence, they buy more quantities in day-ahead, leading to the activation of $k$ slow-start units. The expected real-time price depends on the number of slow-start units $k$ cleared in day-ahead: this is expressed by equation \eqref{eq:VBequilPriceDA}. Furthermore, in order for unit $k$ to be cleared, the demand from virtual bidder should exceed its average cost (cf. the allocation and pricing rules above): this is expressed by equation \eqref{eq:VBequilPriceRT}.
Thus, the equilibrium prices as a function of $k$ are:
\begin{align}
p_{DA} & = P = c_S k + s_S  \label{eq:VBequilPriceDA} \\
\mathbb{E}_{\epsilon} (p_{RT, \epsilon}) & = \frac{1}{2 \overline{D}} \left( \int_0^k c_S \epsilon d\epsilon + \int_k^{\overline{D}} c_F d\epsilon     \right) 
 = \frac{1}{2 \overline{D}}  \left( \frac{c_S}{2} k^2 - c_F k + c_F \overline{D}  \right) \label{eq:VBequilPriceRT}
\end{align} 
Using the no-arbitrage condition, we get:
\begin{align}
0 & = \frac{c_S}{2} k^2 + k (- c_F - 2 c_S \overline{D}) + \overline{D} (c_F - 2 s_S) \label{eq:2orderk*}
\end{align}
Solving the second-order equation, given $0 \leq k \leq \overline{D}$ (notice that one need $c_F \geq 2 s_S$ to have $k^* \geq 0$), leads to
\begin{equation} \label{eq:k*}
k^* = \frac{c_F + 2 c_S \overline{D} - \sqrt{c_F^2 + (2 c_S \overline{D})^2 + 2c_F c_S \overline{D} + 4 c_S s_S \overline{D}}}{c_S}
\end{equation}

Inclusion of virtual trading leads to activation of $k^*$ SSR and reduces the price spread between day-ahead and real-time from $\frac{c_F}{2}$ to zero.
Simple comparative statics on $k^*$ leads to the following results: the lower the competitive advantage of SSR with respect to FSR, the fewer SSR are committed in day-ahead following virtual bidders' strategy.
\begin{myprop} \label{prop:comparativeStaticsK}
Under perfect virtual trading, $k^*$ slow-start resources are committed in day-ahead, such that (focusing on the case $k^* > 0$):
\begin{enumerate}
\item $\partial k^*/\partial s_S < 0$: the higher the fixed cost of slow-start units, the lower $k^*$.
\item $\partial k^*/\partial c_F > 0$: the higher the variable cost of fast-start units, the higher $k^*$.
\item $\partial k^*/\partial c_S < 0$: the higher the variable cost of slow-start units, the lower $k^*$.
\item $\partial k^*/\partial s_F = 0$: $k^*$ is insensitive to the fixed cost of fast-start units.
\end{enumerate}
\end{myprop}

The equilibrium expected side payments in the market that includes virtual trading are then as follows:
\begin{align}
& SP_{DA} = 0 \label{eq:VBequilSPDA} \\
& \mathbb{E}_{\epsilon} (SP_{RT,\epsilon}) = \frac{1}{2 \overline{D}} \int_k^{\overline{D}} s_F (\epsilon-k) d\epsilon = \frac{s_F \overline{D}}{4} - \frac{s_F k}{2} \left( 1 - \frac{k}{2 \overline{D}} \right) \label{eq:VBequilSPRT}
\end{align}

As far as side payments in day-ahead are concerned, we see that there is a discrete effect: introducing virtual trading, in our model, fully eliminates side payments. 
Regarding the side payments in real-time, the two main observations are: they are reduced but not eliminated, and they are linearly increasing with $s_F$, the fixed costs of FSR.

\begin{myprop}[Side payments 1] \label{prop:sidePayments}
Perfect virtual trading decreases expected side payments both in day-ahead and in real-time.
\end{myprop}

\begin{myprop}[Side payments 2] \label{prop:sidePayments2}
Real-time side payments are linearly increasing with $s_F$.
\end{myprop}

This is illustrated in Figure \ref{fig:SP_k}. There, we see $k^*$ reduces side payments with respect to the case of no virtual trading (cf. Proposition \ref{prop:sidePayments}). The figure also shows that the side payments scale linearly with $s_F$ (cf. Proposition \ref{prop:sidePayments2}): given $k^*$, whose value does not change with $s_F$, going from $s_F=0.5$ to $s_F=1$ implies the same increase in side payments as going from $s_F=1$ to $s_F=1.5$.

These results follow virtual trading bidding behaviour, which has two main features: (i) they anticipate real-time conditions and internalise the fixed costs of the physical resources (cf. equation \eqref{eq:VBequilPriceDA}) and (ii) they set the price in DA. Together, this leads to a day-ahead price that reflects the fixed costs, thus side payments drop to zero ; and it also entails better commitment decisions in day-ahead that lead to fewer activations of fast start resources in real-time, thus less side payments in real-time.

Let us make two remarks. 
First, real-time side payments are not eliminated by virtual trading. Because the physical real-time market involves non-convexities in the production, the uniform price alone might not clear the market and still requires side payments.
One way---which we will not discuss further in our paper---to improve these issues is by amending the auction pricing rule itself, deviating from marginal pricing, to reflect some of the fixed costs into the price signal. In our model, this would for instance mean setting the price to $c_F + s_F$ instead of $c_F$ when FSR are activated in real-time. This pricing rule would eliminate the need for side payments in real-time, in our model. This is the spirit of the so-called ``fast-start pricing approaches'' discussed in the literature and implemented by some US ISOs.
Our goal is thus not to argue that these discussions on the right pricing rule are irrelevant---it remains important and it is compatible with virtual trading---but to highlight the implications of virtual trading itself.
Second, virtual trading fully eliminates day-ahead side payments. One caveat though is that our model assumes a competition \textit{\`a la Bertrand} with an infinite quantity of virtual traders, bidding at a price $P$.
Following our pricing rule definition, it implies the day-ahead price is set by $P$. If instead there was a limited volume of virtual traders, the demand curve could intersect the supply curve vertically (cf. Figure \ref{fig:modelEnvMechanism}). As a consequence, SSR $k$ would set the price to $c_S k$, which would entail some side payments to cover the fixed costs.
Besides, our model also neglects the multi-period nature of the market, inter-temporal constraints and most unit commitment constraints that are encountered in actual markets.
Thus, although our model predicts that virtual trading fully eliminates DA side payments, this result should be viewed as a theoretical benchmark, and observed outcomes may depart from it for several reasons (we come back to this in section \ref{sec:empiricalStrat}).

\begin{figure}
\centering
\includegraphics[width=0.8\textwidth]{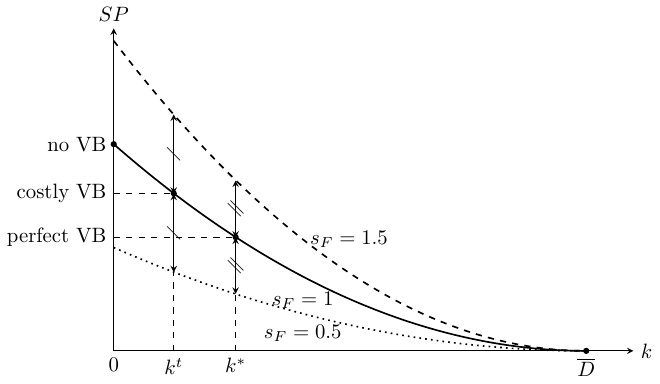}
\caption{Side payments as a function of $k$ (with $d=2$, $\overline{D}=1$, $c_S=1$, $s_S=0.5$, $c_F=2$, $t=0.25$) and three values of $s_F$: $s_F=0.5$, $s_F=1$ and $s_F=1.5$.}
\label{fig:SP_k}
\end{figure}

Let us turn to the analysis of the expected costs.
With virtual trading, the total expected cost as a function of $k$ is:
\begin{align}
& Cost(k) = s_S d + s_S k + \frac{1}{2 \overline{D}} \left( \int_0^k \frac{c_S}{2} \epsilon^2 d\epsilon + \int_k^{\overline{D}} \frac{c_S k^2}{2} + (c_F + s_F) (\epsilon - k) d\epsilon    \right) \\
&= \underbrace{s_S d + \frac{\overline{D}}{4}(c_F + s_F)}_{\text{cost without VB}} + \underbrace{s_S k}_{\text{sunk cost}} - \underbrace{\frac{k^3 c_S}{6 \overline{D}} +\frac{k^2(c_F+s_F)}{4 \overline{D}} + \frac{k^2 c_S}{4} - k\frac{(c_F + s_F)}{2}}_{\text{expected cost savings}} \label{eq:cost_withVB}
\end{align}
The cost is a sum of three terms: (i) the cost without virtual virtual trading, (ii) the sunk cost implied by committing $k$ SSR and (iii) the cost savings entailed by committing $k$ SSR.
Comparing expression \eqref{eq:cost_withVB} with expression \eqref{eq:costNoVB}, direct mathematical derivations lead to the following observation:
\begin{myprop}[Efficiency 1] \label{prop:efficiency1}
Perfect virtual trading improves cost efficiency compare to no virtual trading.
\end{myprop}

Although virtual trading improves efficiency, it does not reproduce the first-best commitment decisions. The reason is that the virtual traders only react to the day-ahead and real-time price signals. With non-convexities, the price does not convey all the relevant cost information---mainly, it overlooks the fixed costs. For example, $k^*$ does not depend on $s_F$: $\partial k^*/\partial s_F = 0$. While clearly, if $s_F \rightarrow \infty$, assuming $s_S$ remains bounded, it would be optimal to commit all slow-start units, thus having $k = \overline{D}$.
\begin{myprop}[Efficiency 2] \label{prop:efficiency2}
Perfect virtual trading does not reproduce the first-best efficient day-ahead commitment plan $k^{**}$: it entails lower commitments of SSR than optimal, thus $k^* < k^{**}$.
\end{myprop}

Propositions \ref{prop:efficiency1} and \ref{prop:efficiency2} are illustrated on Figure \ref{fig:costCurvek}. Virtual trading entails committing $k^*$ SSR in day-ahead which reduces the overall costs compared to the situation without virtual trading (Proposition \ref{prop:efficiency1}). Yet, it does not lead to the first-best commitment plan $k^{**}$ (the optimum of social planner's two-stage stochastic optimization problem) but leads to under-commitment, thus $k^* < k^{**}$ (Propositions \ref{prop:efficiency2}).\footnote{Notice that $s_F=0$ does not imply $k^*=k^{**}$: the fixed costs $s_S$ also distort incentives as they are reflected in the day-ahead price through the virtual traders but not in the real-time prices. More surprisingly, $s_F=s_S=0$ would not imply $k^*=k^{**}$ either, in our model.
If $s_S=0$ then $k^{**}=\overline{D}$: there is no stranded cost involved in committing SSR thus it is optimal to commit them all. 
While $k^* < \overline{D}$ for $s_S=0$, as $k^* = \overline{D}$ would imply the day-ahead price is always greater than the real-time price, thus the virtual traders would not achieve this full commitment outcome.
The reason of this failure is that the commitment decisions in day-ahead entail some inflexibility in real-time. Even with $s_F=s_S=0$, the commitment leads to virtual trading that fails to be fully efficient.}

\begin{figure}
\centering
\includegraphics[width=0.8\textwidth]{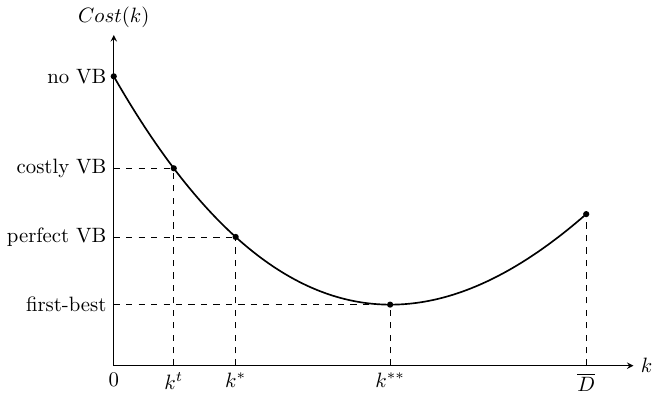}
\caption{Total expected costs as a function of $k$ (with $d=2$, $\overline{D}=1$, $c_S=1$, $s_S=0.5$, $c_F=2$, $s_F=1$, $t=0.25$).}
\label{fig:costCurvek}
\end{figure}

Let us finally remark that, in \textit{some} scenarios, virtual trading may reduce efficiency. For instance, for $\epsilon = -\overline{D}$, the cost would have been lower without financial trading. However, what matters is that \textit{in expectation} financial trading improves efficiency.\footnote{\cite{PJM2015} argued that virtual trading can deter efficiency based on some deterministic examples that assume an exogenous bidding behavior for financial traders, and as we argue, what matters is what happens \textit{in expectation} and \textit{at equilibrium, and} we further discuss this in section \ref{sec:ccl}.} 

\paragraph*{Equilibrium with costly virtual trading.}
Let us consider the case of ``costly'' virtual trading where a transaction cost $t$ is introduced. Virtual traders competing \textit{\`a la Bertrand} now leads to the following arbitrage condition: $\mathbb{E}_{\epsilon} (p_{RT, \epsilon}) - p_{DA} = t$, thus the price spread must exactly refund the transaction fee.
Using a similar reasoning as previously, one can compute the SSR commitments $k^t$ resulting from virtual trading with a transaction cost $t$. This corresponds to the $k^t \in [0, \overline{D}]$ solving:
\begin{equation} \label{eq:kcondition_with_t}
0 = \frac{c_S}{2} k^2 + k (- c_F - 2 c_S \overline{D}) + \overline{D} (c_F - 2 s_S - 2 t) 
\end{equation}
Condition \eqref{eq:kcondition_with_t} is the conterpart of \eqref{eq:2orderk*} when transaction cost $t$ is included. Comparing \eqref{eq:kcondition_with_t} and \eqref{eq:2orderk*} directly leads to $k^t \leq k^*$ (notice that one need $t \leq \frac{c_F - 2s_S}{2}$ to have $k_t \geq 0$), and $\frac{\partial k^t}{\partial t} < 0$, thus the higher the fee the lower $k^t$. A transaction fee implies that less financial trading volume gets cleared by the auction, and the higher the fee $t$ the lower the cleared volume. 

All the equilibrium results derived above as functions of $k$ still hold, substituting $k$ with $k^t$ (the prices in equations \eqref{eq:VBequilPriceDA}--\eqref{eq:VBequilPriceRT}, the side payments in equations \eqref{eq:VBequilSPDA}--\eqref{eq:VBequilSPRT} and the cost in equation \eqref{eq:cost_withVB}). 
Concretely, the side payments are reduced compared to no-virtual trading but this improvement shrinks as $t$ increases. This is illustrated on Figure \ref{fig:SP_k}.
Similarly, the cost is improved compared to no-virtual trading but this improvement shrinks as $t$ increases. This is illustrated on Figure \ref{fig:costCurvek}.

\begin{myprop}[VB with transaction cost $t$] \label{prop:transactionCost}
Virtual trading with transaction cost $t$ leads to $k^t \leq k^*$ such that:
\begin{itemize}
\item $Cost(k^*) \leq Cost(k^t) \leq Cost(0)$
\item $SP(k^*) \leq SP(k^t) \leq SP(0)$
\end{itemize}
\end{myprop}

\paragraph*{Main takeaways.} 
To sum-up, there are three main takeaways from our model. (i) First, virtual trading improves price convergence between DA and RT, although a transaction fee $t$ would entail a persistent price difference of $t$. (ii) Second, more virtual trading leads to better commitment decisions in DA (more activations of slow-start units), which lowers the expected costs. Both the improvement of price convergence and the reduction of total cost have been analysed in the literature (see, for instance, \cite{mather2017} for an analytical model, \cite{kazempour2017a,kazempour2017b} for numerical simulations relying on a comprehensive electricity auction model and \cite{jha2023} for empirical evidences from CAISO market).
(iii) The third effect---and most important for our analysis---is that higher virtual trading mitigates the need for side payments, thus it leads to a uniform pricing outcome closer to an equilibrium.
Large side payments means some cost information is not well reflected in the market price and is instead handled by out-of-market payments, while low side payments means the uniform price better signals the overall cost information. 
Here, there are two distinct effects.
In DA, virtual trading set the price leading to prices that endogenise the fixed costs of SSR which eliminates the day-ahead side payments. 
In RT, as virtual traders have anticipated the risk of FSR activation, the likelihood of lumpy FSR activations is reduced which mitigates the associated side payments.

\begin{figure}[t!]
    \centering
    \begin{subfigure}[b]{0.5\textwidth}
        \centering
        \includegraphics[width=\textwidth]{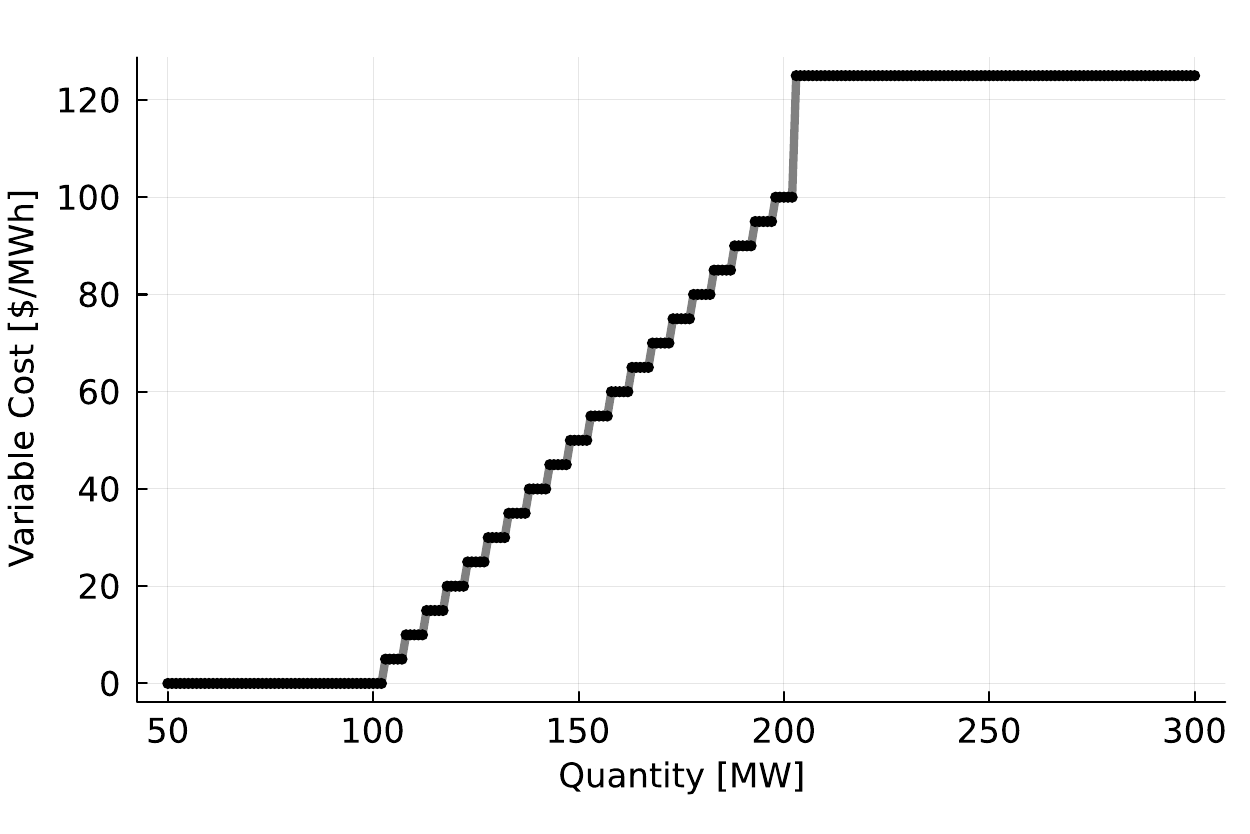}
        \caption{Merit order curve}
        \label{fig:example_MO}
    \end{subfigure}%
    ~ 
    \begin{subfigure}[b]{0.5\textwidth}
        \centering
        \includegraphics[width=\textwidth]{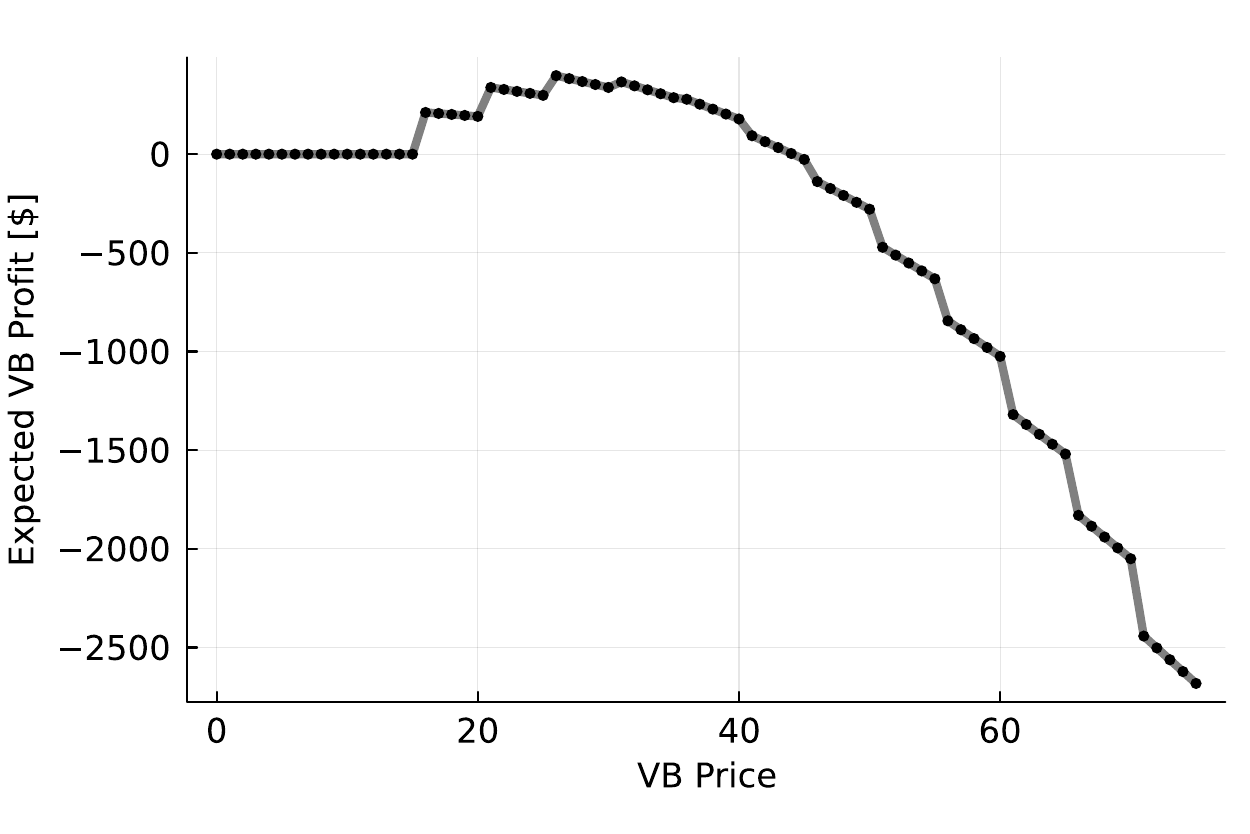}
        \caption{Virtual bidders profit}
        \label{fig:example_VBprofit}
    \end{subfigure}
    \begin{subfigure}[b]{0.5\textwidth}
        \centering
        \includegraphics[width=\textwidth]{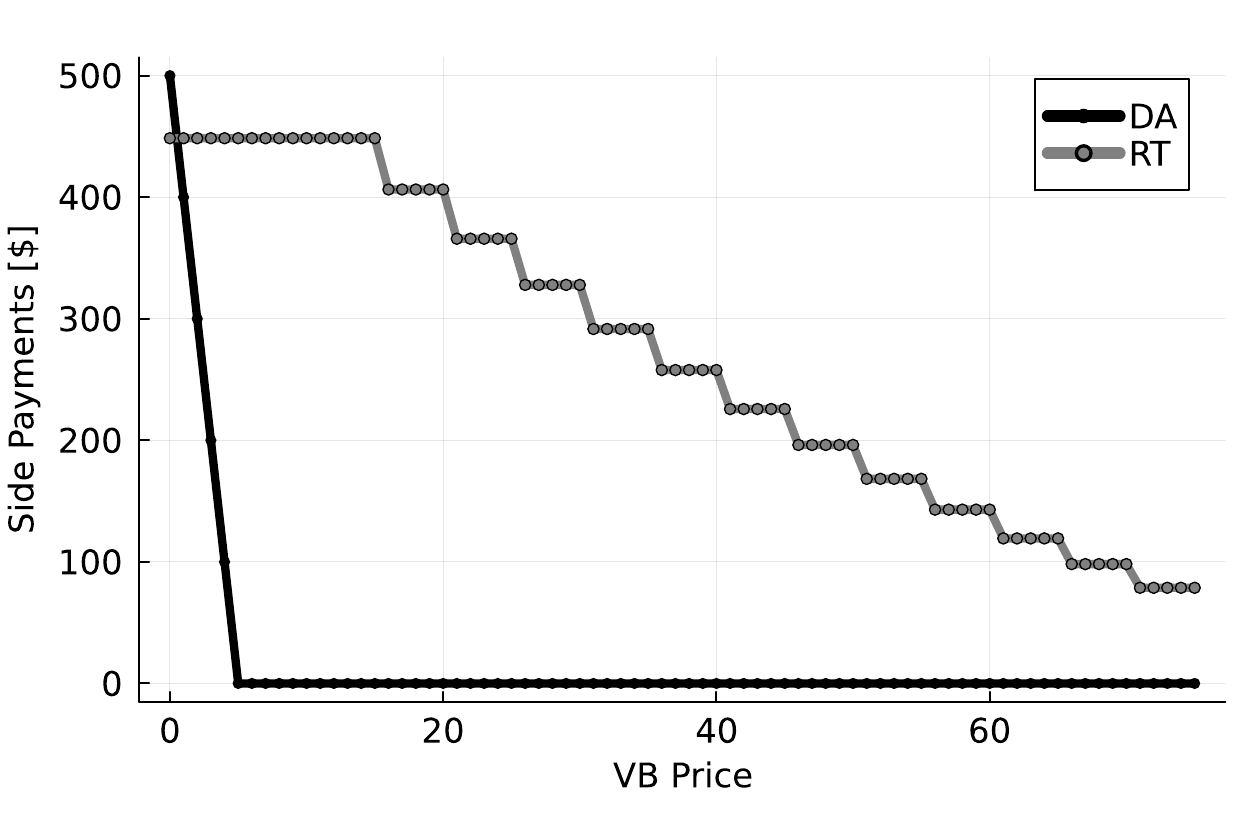}
        \caption{Side Payments}
        \label{fig:example_SP}
    \end{subfigure}%
    ~ 
    \begin{subfigure}[b]{0.5\textwidth}
        \centering
        \includegraphics[width=\textwidth]{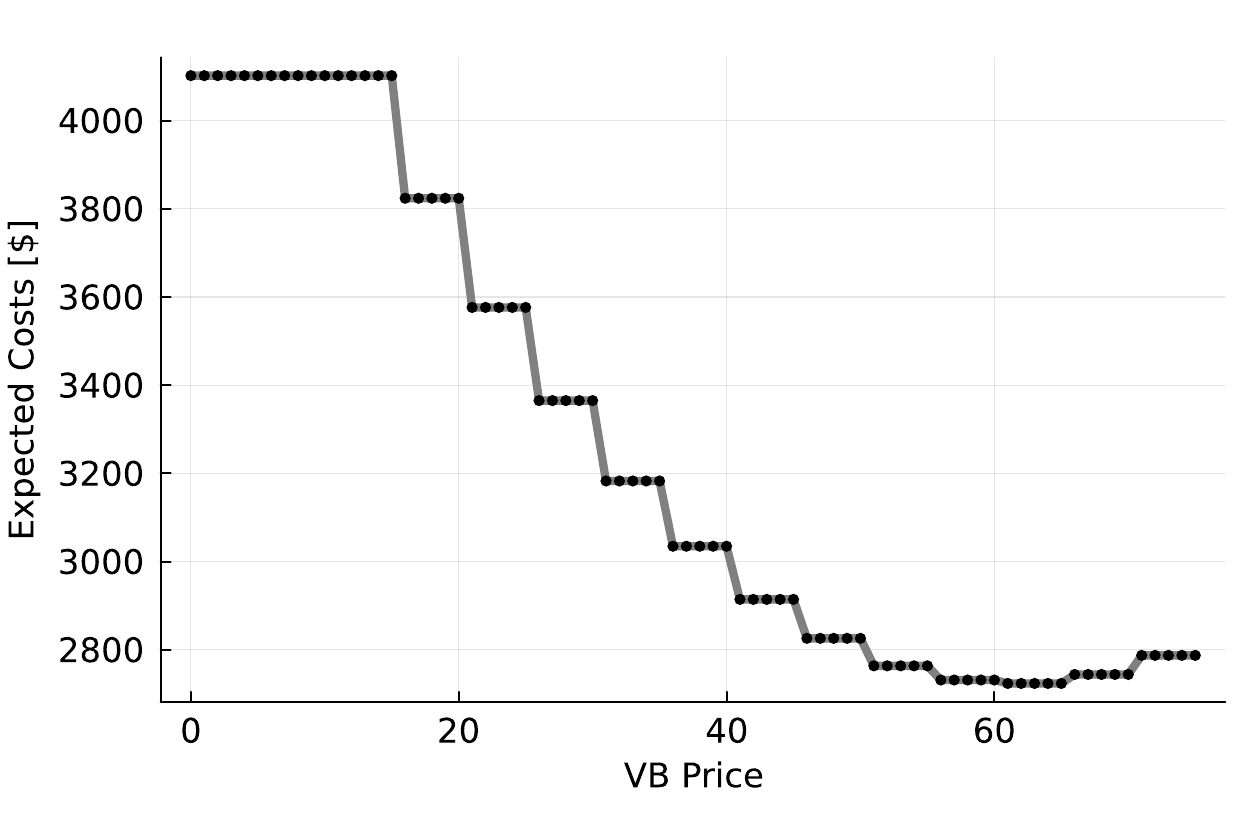}
        \caption{Total expected costs}
        \label{fig:example_Costs}
    \end{subfigure}
    \caption{Results of the discrete example.}
    \label{fig:example_results}
    \label{fig:}
\end{figure}

\paragraph*{A discrete example} 
Let us consider the following non-convex electricity market. There are 20 SSR and 20 FSR, with a capacity of 5MW each and a minimum production limit of 2 MW each. The fixed cost of SSR are \$50, while it is \$85 for the FSR. The variable cost of production is $i \times 5$\$/MWh for $i \in \{1, 20\}$ for the SSR and 125\$/MWh for the FSR.
There is also a baseload SSR unit of 100MW with variable cost 0 and fixed cost equal to \$500.
The merit order curve of this example is shown on Figure \ref{fig:example_MO}.
$d=100$MW and the demand shocks $\epsilon$ are uniformly distributed in the set $\{ 2i \vert i \in -50 \ldots 50 \}$. We consider the case of perfect virtual trading with a large volume of VB who choose a bid-price strategy $P$.
As in our model, the auctioneer clears the surplus-maximizing allocation and uses marginal pricing complemented by side payments, both for the day-ahead and real-time auctions.
This instance is fairly close to the settings of our model, except that the previous continuum of agents is replaced by a finite number of lumpy producers.

Figure \ref{fig:example_results} shows the market outcomes, as a function of the bid price strategy of the virtual traders. 
No virtual trader corresponds to $P^*=0$, as no virtual trader would get cleared at this price.
Perfect virtual trading corresponds to full arbitrage between day-ahead and real-time prices, leading to zero profit for the virtual traders. Here, this corresponds to the strategy $P^* \approx 44$ (cf. Figure \ref{fig:example_VBprofit}). This strategy entails dispatching 30MW of SSR (thus 6 SSR commitments) in day-ahead.
As in our model, this leads to cost reduction compared to no virtual trading (cf. Proposition \ref{prop:efficiency1}), as observed on Figure \ref{fig:example_Costs}, although it does not quite minimize the expected costs (cf. Proposition \ref{prop:efficiency2}). As observed on Figure \ref{fig:example_SP}, the day-ahead side payments are null and the real-time side payments decrease---almost cut by two---compared to no virtual trading (cf. Proposition \ref{prop:sidePayments}).

The outcome in presence of a transaction cost $t$ can also be deduced from the figures. Instead of zero profit, a transaction cost $t$ implies the virtual traders should make a per-MW profit of $t$, thus shifting their bidding strategy to the left on Figure \ref{fig:example_VBprofit}, increasing both costs and side payments compared to the perfect virtual trading case (cf. Figures \ref{fig:example_Costs} and \ref{fig:example_SP}), as expected from Proposition \ref{prop:transactionCost}.
These simulations also hint on what would happen in case of imperfect competition among virtual bidders. 
The price strategy maximizing virtual bidders' profit in Figure \ref{fig:example_VBprofit} ($P^*=26$) should be read as the outcome if virtual bidders behave as a monopsony instead of behaving \textit{\`a la Bertrand}. 
We observe that imperfect virtual trading---either monopsony virtual bidder or virtual trading with transaction cost $t$---improves both costs and side payments compared to no virtual bidders, although it does not perform as well as perfect virtual trading.

We ran some more simulations introducing heterogeneity in the portfolio of power plants. 
The results confirm the previous observations and are reported in the appendix \ref{sec:appendixMoreExamples}.

\section{Empirical strategy and data} \label{sec:empiricalStrat}

\paragraph*{Methodology.} In order to test the predictions of our model, we analyze the effect of a policy that impacted the transaction fees paid by some of the virtual traders in PJM Interconnection market. 
On November 1, 2020, a transaction fee was introduced on the UTC virtual bids \citep{MonitoringAnalytics2020,FERC2020}. 
INC and DEC virtual bids were already subject to transaction fees while UTC bids were not. This was judged ``unjust, unreasonable, and unduly preferential''\footnote{``We find that PJM's current uplift allocation rules are unjust, unreasonable, and unduly preferential because they do not allocate uplift to UTCs. Accordingly, we direct PJM to submit, in a compliance filing within 45 days from the date of this order, a replacement rate that revises PJM's current uplift allocation rules to allocate uplift to UTCs in a manner that treats a UTC, for uplift allocation purposes, as if the UTC were equivalent to a DEC at the sink point of the UTC. As a result, under the replacement rate, UTCs will be allocated both real-time uplift and day-ahead uplift.'' \citep{FERC2020}} by the FERC, the US energy regulator, who decided that the UTC bids should pay day-ahead and balancing operating reserve charges equivalent to a DEC at the UTC sink point \citep{FERC2020}.
The transaction fee applies to cleared quantities (as in our model of section \ref{sec:model}).

\begin{figure}
\centering
\includegraphics[width=0.9\textwidth]{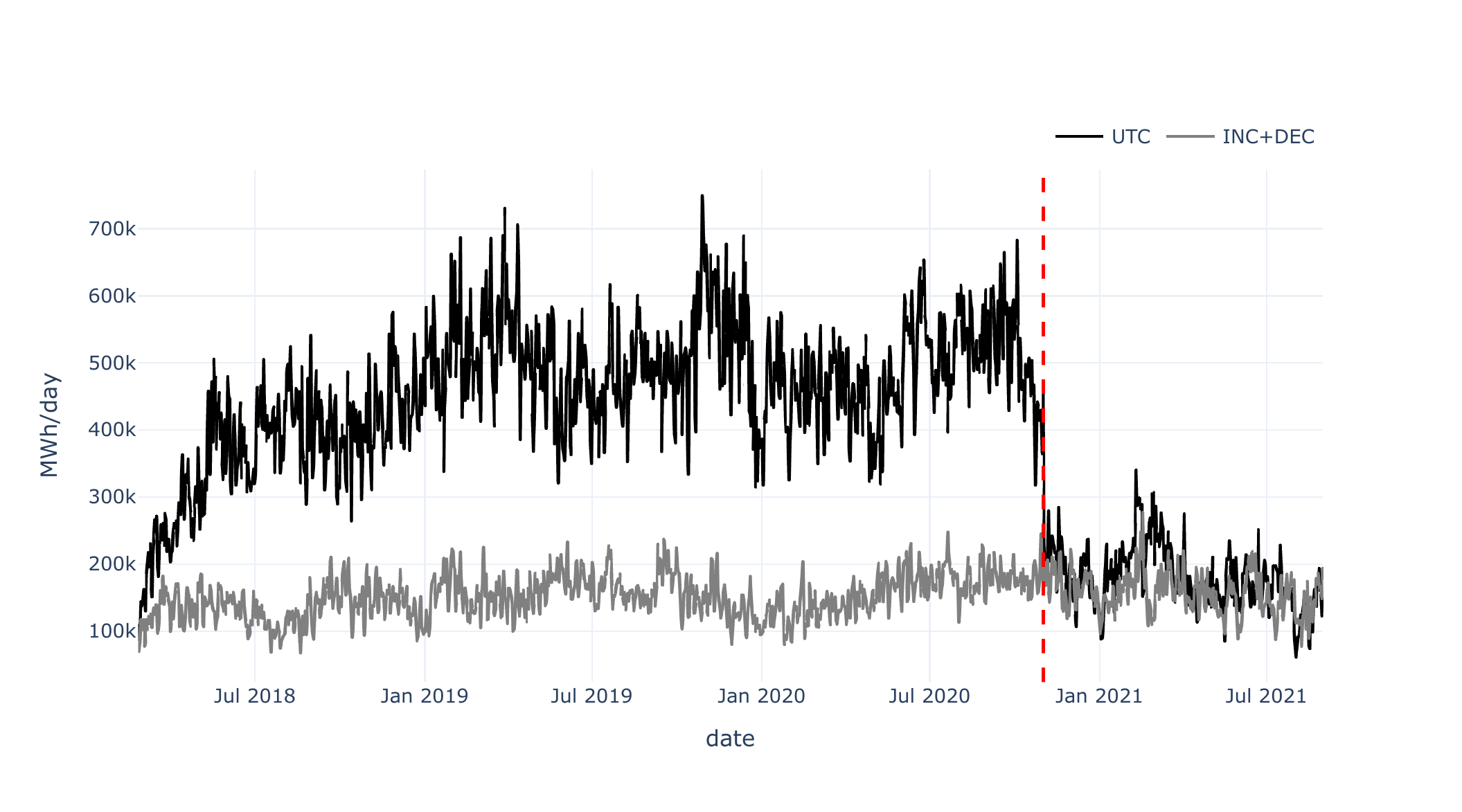}
\caption{PJM cleared virtual bids (INC, DEC and UTC). The vertical line, on November 1, 2020, corresponds to the introduction of transaction fees on UTC bids.}
\label{fig:PJM_virtual_bids_treatment}
\end{figure}

Following the implementation of this policy, the cleared volume of UTC virtual bids dropped significantly. This is shown on Figure \ref{fig:PJM_virtual_bids_treatment} where the vertical red line corresponds to the implementation of the new transaction fee. We see this transaction fee did not substantially affect the cleared volumes of other virtual bids (INC and DEC) but led to a significant decrease of cleared UTC bids (it is worth noting that, at the time, UTC represented the largest share of total financial trading volume, cf. Table \ref{tab:virtualBidsStat}).
Overall, the total hourly average cleared volume of virtual bids (INC, DEC and UTC) went from 25GW/h to 14GW/h following the introduction of the transaction fee.

\begin{table}
\centering
\setlength{\tabcolsep}{5pt}
\resizebox{\textwidth}{!}{
\begin{tabular}{lllllll}
\toprule
& \multicolumn{2}{@{}c@{}}{INC} & \multicolumn{2}{@{}c@{}}{DEC} & \multicolumn{2}{@{}c@{}}{UTC} \\
& Pre-Treat. & Post-Treat. & Pre-Treat. & Post-Treat. & Pre-Treat. & Post-Treat.  \\ 
\cmidrule(lr){2-3}\cmidrule(lr){4-5}\cmidrule(lr){6-7}
Mean &  64.3 & 55.9  & 87.2 & 102.2 & 456.6 &  180.3  \\ 
Median & 62.5 & 55.3 & 83.3 & 99.8 & 462.1 & 175.3  \\ 
Stand. Dev. &  16.5 &  15.7 &  28.6 & 27.8  &  101.9 & 46.5  \\ 
Max &  126.7 &  104.5 &  181.0 & 228.9 &  749.8 & 341.1  \\ 
\# observations &  978 & 298  &  978 & 298  &  978 & 298  \\ 
\bottomrule
\end{tabular}}
\caption{Daily cleared volume of virtual bids [GWh/day], pre- and post-treatment.} 
\label{tab:virtualBidsStat}
\end{table}

This drop in the cleared volume is aligned with economic expectations and with our model predictions: a transaction fee impacts the profit-maximizing problem of the arbitragers, who offset their bid to cover the fee, leading to a reduction of the cleared volume.
Our model predicts that the decline in financial trading volume should coincide with an increase in side payments (cf. Proposition \ref{prop:transactionCost}).
Did the inclusion of a transaction fee on November 1, 2020, and the associated sharp drop in virtual trading volume, indeed lead to an increase in side payments, as predicted by our model?

To analyze the effect of this policy, we leverage the public market data of PJM\footnote{except the Henry hub gas price time series which is sourced from EIA, the US Energy Information Administration.}. PJM publishes comprehensive market data including the detailed physical bids, the virtual bids, the load, the self-scheduling volume, data on side payments, data related to grid congestion and outages, etc.
We focus on the period from end of February 2018 to September 2021. This period is selected for it is stable in the market rules, aside from the policy change on UTC transaction fee.
On February 20, 2018, FERC issued an order which limit the number of bidding points at which virtual bids may be submitted by market participants (Docket No. ER18-88, \cite{FERC2018}). This changed the scope and nature of virtual bids, thus we start our dataset after February 20, 2018.
Besides, on September 1, 2021,  PJM implemented a new ``fast-start pricing'' rule on both the day-ahead and real-time markets.\footnote{Extensive discussions about ``fast-start'' pricing at PJM started in 2017 \citep{PJM2017}, following the trend initiated by the US Regulatory Commission which launched a consultation about price formation in power auctions in 2014 \citep{FERC2014}. This new pricing approach, sometimes referred to as ``extended LMP'', consists of taking the linear programming (LP) relaxation of the unit commitment problem. The main difference between fast-start pricing and extended LMP discussed in the scientific literature, is that fast-start pricing only consider the LP relaxation for a limited set of resources, namely those that are ``fast start''. Roughly said, fast-start pricing deviates from marginal pricing for it partly reflects the start-up costs, and other operational fixed costs, in the price signal, cf. the discussion in section \ref{sec:model} about including $s_F$ in the price signal.}
Up to September 2021, PJM relied on the classic marginal cost pricing rule, as in our model. 
Concretely, this new pricing rule deviates from marginal pricing for it allows the fixed non-convex costs of fast-start resources to be reflected in the price signal. 
This is a disruptive change in the pricing rule, thus we end our dataset on September 1, 2021.\footnote{Note that, the resources qualifying as ``fast-start'' for this pricing rule stand for a small share of all resources. Although this pricing policy is disruptive in theory, because of the limited number of resources which are effectively treated as ``fast-start'', the effect of fast-start pricing was expected to be moderate, according to simulations conducted by PJM in 2019 \citep{PJM2019faststartprice}.
Ex post, it appears the magnitude of its impact has been almost nonexistent in day-ahead and moderate in real-time, although some months exhibit several percentage point difference in prices between fast-start pricing and marginal pricing \citep{MonitoringAnalytics2022,MonitoringAnalytics2025faststart} (with the caveat that these \textit{simulated} marginal prices would have been different if they were actually used for settlement, since the bidding behaviour would likely have been different as well).}
Overall, our database includes 1,276 days, excluding 12 days between February 20, 2018, and September 1, 2021, because of missing data.


Our main variables of interest are the day-ahead and real-time side payments paid by PJM, i.e. the counterpart of $SP_{DA}$ and $SP_{RT}$ in our model of section \ref{sec:model}. 
As it should be clear by now, side payments are used as a measure of distance to equilibrium: they are the discriminatory payment made to market participants to compensate for their losses at the market price. PJM provides data of the daily amount of side payments (these payments indeed do not have a finer, hourly, granularity).
Table \ref{tab:upliftsData} provides descriptive statistics of day-ahead and real-time daily side payments. Let us highlight two main observations. First, most of the side payments due to non-convexities are paid in real-time, as already observed on Figure \ref{fig:PJMuplifts}: average real-time side payments are about four times larger in real-time than in day-ahead. This should mostly be read as side payments related to the activation of fast-start resources in real-time, as in our model.
Second, 22\% of the days exhibit zero day-ahead side payments, which means about 6 days per month. Figure \ref{fig:PJM_zero_upliftDA} shows the number of days per month with zero day-ahead side payments.

\begin{table}[t]
\centering
\setlength{\tabcolsep}{5pt}
\resizebox{\textwidth}{!}{
\begin{tabular}{lllllll}
\toprule
& \multicolumn{3}{@{}c@{}}{Day-ahead} & \multicolumn{3}{@{}c@{}}{Real-time} \\
& All Days & Pre-Treat. & Post-Treat. & All Days & Pre-Treat. & Post-Treat. \\ 
\cmidrule(lr){2-4}\cmidrule(lr){5-7}
Mean & 47,620 & 51,925 & 33,494 & 215,705 & 167,351 & 374,400 \\ 
Median & 7582 & 7471 & 7595 & 109,672 & 90,588 & 220,052  \\ 
Stand. Dev. & 86,722 & 92,403 & 62,746 & 296,569 & 210,532 & 445,874 \\ 
Max & 958,693 & 958,693 & 646,474 & 2,604,937 & 1,789,976 & 2,604,937 \\ 
\% of 0 & 22\% & 24\% & 15.7\% & 0.08\% & 0.1\% & 0\% \\ 
\# observations & 1276 & 978 & 298 & 1276 & 978 & 298 \\ 
\bottomrule
\end{tabular}}
\caption{Day-ahead and real-time side payments [\$/day], pre- and post-treatment.} 
\label{tab:upliftsData}
\end{table}

\begin{figure}
\centering
\includegraphics[width=0.9\textwidth]{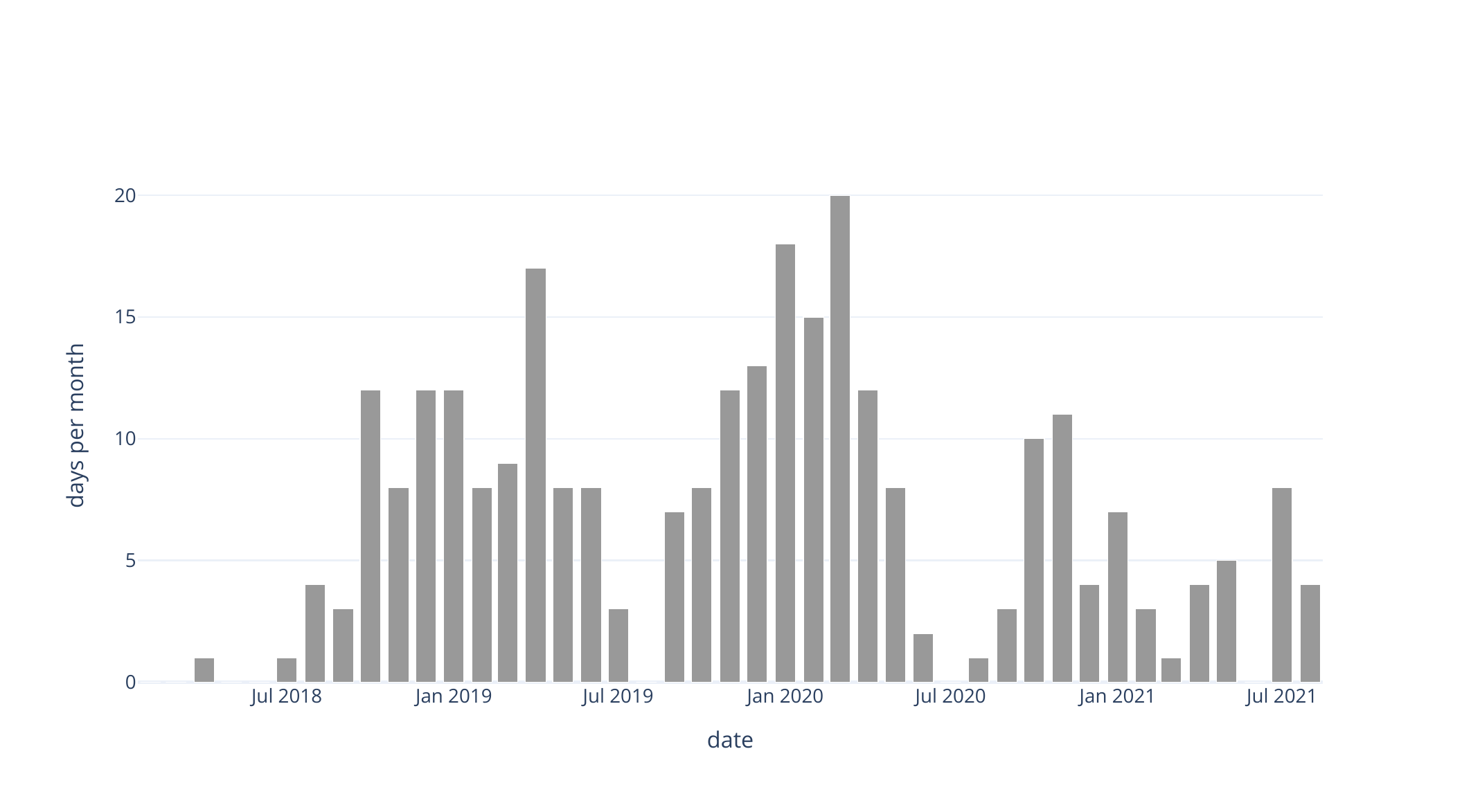}
\caption{Days per month with zero side payments in day-ahead in PJM market.}
\label{fig:PJM_zero_upliftDA}
\end{figure}

To estimate the impact of the policy on the side payments, we estimate two different models: one for day-ahead side payments and one for real-time side payments.

\paragraph*{Real-time model.}
For the real-time side payments, we estimate the following ordinary least squares (OLS) regression:
\begin{equation} \label{eq:OLS_RT}
SP_{RT,t} =  \alpha + \beta_1 \text{Treatment}_t + \beta_2 s_{F,t} + \beta_3 loadShock_t + \eta \textbf{X}_t + u_t
\end{equation}
where $SP_{RT,t}$ corresponds to the real-time side payments for each day $t$, Treatment$_t$ is the indicator variable representing the introduction of a transaction fee on virtual bids on November 1, 2020, $s_{F,t}$ are the fixed costs of FSR, $loadShock_t$ represents the load shocks and $\textbf{X}_t$ are the control variables.
The main parameter of interest is $\beta_1$ which measures the treatment effect, thus by how much did the transaction fee increased the amount of real-time side payments. From our model we expect $\beta_1 > 0$: a transaction fee increases the amount of side payments needed to sustain an equilibrium.

In addition to the impact of financial trading, our model also points to a direct effect of both FSR fixed costs (i.e. $s_F$ in our model) and load shocks (i.e. $\epsilon \in [-\overline{D}, \overline{D}]$ in our model) on side payments. These effects are captured by the parameters $\beta_2$, which is the effect of the fixed costs of FSR, and $\beta_3$, which is the effect of the load shock between day-ahead and real-time.
We briefly comment on the FSR data and load shock data before turning to the control variables.

\textsc{FSR bids data---}PJM publishes comprehensive bid data of dispatchable resources, including, for each resource, the hourly offered curve, start-up and no-load costs, maximum and minimum production limit, start time and minimum up and down time.
We use these bid data to construct a time series of the fast-start resource fixed costs (the counterpart of parameter $s_F$ in our model of section \ref{sec:model}).
In PJM, eligible fast-start resources should have a total time to start and minimum run time less than or equal to one hour \citep{PJM2025b}.
Using these criteria and focusing on non-convex resources (with positive fixed costs), leads to about 180 units that are ``fast-start''.\footnote{The public figures available speak about 200 to 300 resources \citep{MonitoringAnalytics2025faststart}.}
These resources have a capacity mostly between 0 and 100MW, with a median at 43MW (cf. the distribution in Figure \ref{fig:FSRdata} in appendix).
Each one of these resources has a fixed cost which corresponds to the sum of the start-up cost and the no-load cost, and which is mostly between 0 and 10,000\$ (cf. the distribution in Figure \ref{fig:FSRdata} in appendix).
We aggregate the fixed costs of these resources by taking a weighted average, where the plant capacity is used as weights: the weighted average FSR fixed cost is around 4000\$.
Figure \ref{fig:PJMgasPriceAndFixedCosts} shows the output time series as well as the Henry hub gas price, which explains some but not all the FSR fixed costs variations. 
This time series corresponds to $s_{F,t}$ in equation \eqref{eq:OLS_RT}. Parameter $\beta_2$ measures by how much the real-time side payments increase when the fixed costs of FSR increases. From our model (cf. Proposition \ref{prop:sidePayments2}), we expect $\beta_2>0$: more fixed costs, typically refunded by side payments, naturally increases the amount of side payments.

\begin{figure}[t]
\centering
\includegraphics[width=0.9\textwidth]{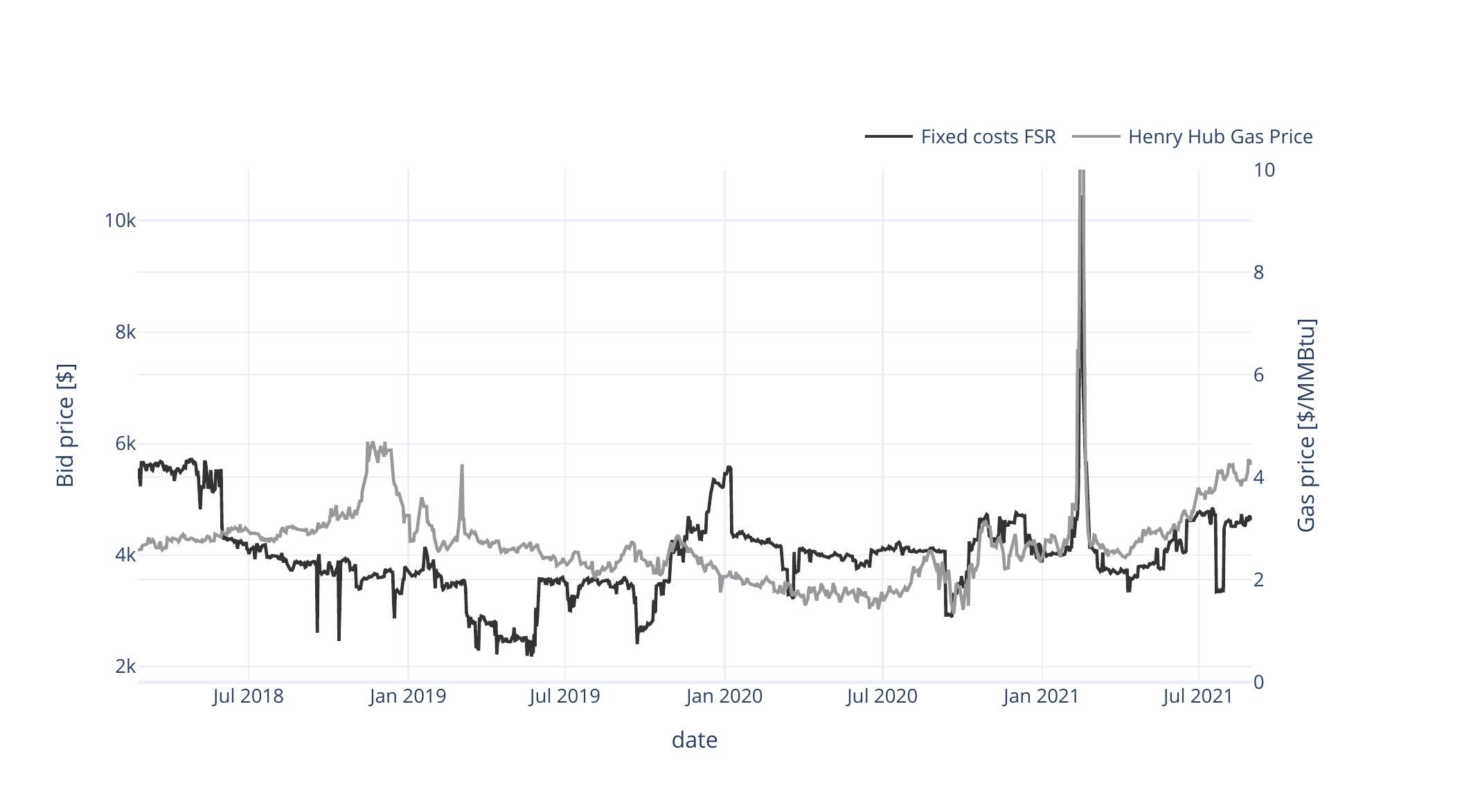}
\caption{Henry Hub gas price and PJM fixed costs of fast-start resources.}
\label{fig:PJMgasPriceAndFixedCosts}
\end{figure}

\textsc{Load shocks---}We compute load shocks (or load forecast error) as the hourly square difference between DA and RT loads.
This time series corresponds to $loadShock_t$ in equation \eqref{eq:OLS_RT}. 
We interpret it as the counterpart to the load shock $\epsilon \in [-\overline{D}, \overline{D}]$ in our model of section \ref{sec:model}.
Figure \ref{fig:PJM_load_val_std_error} shows the monthly average of the load shocks and compare it with the daily load and daily load standard deviation. 
The average daily load error is 12GW, thus 500MW/h, meaning 0.5\% of the hourly load.
We observe that load shocks tend to be higher during the summer months, when both the load and load variation are higher.
We expect $\beta_3>0$: larger shocks implies more FSR activations thus larger side payments.

\begin{figure}[t]
\centering
\includegraphics[width=0.9\textwidth]{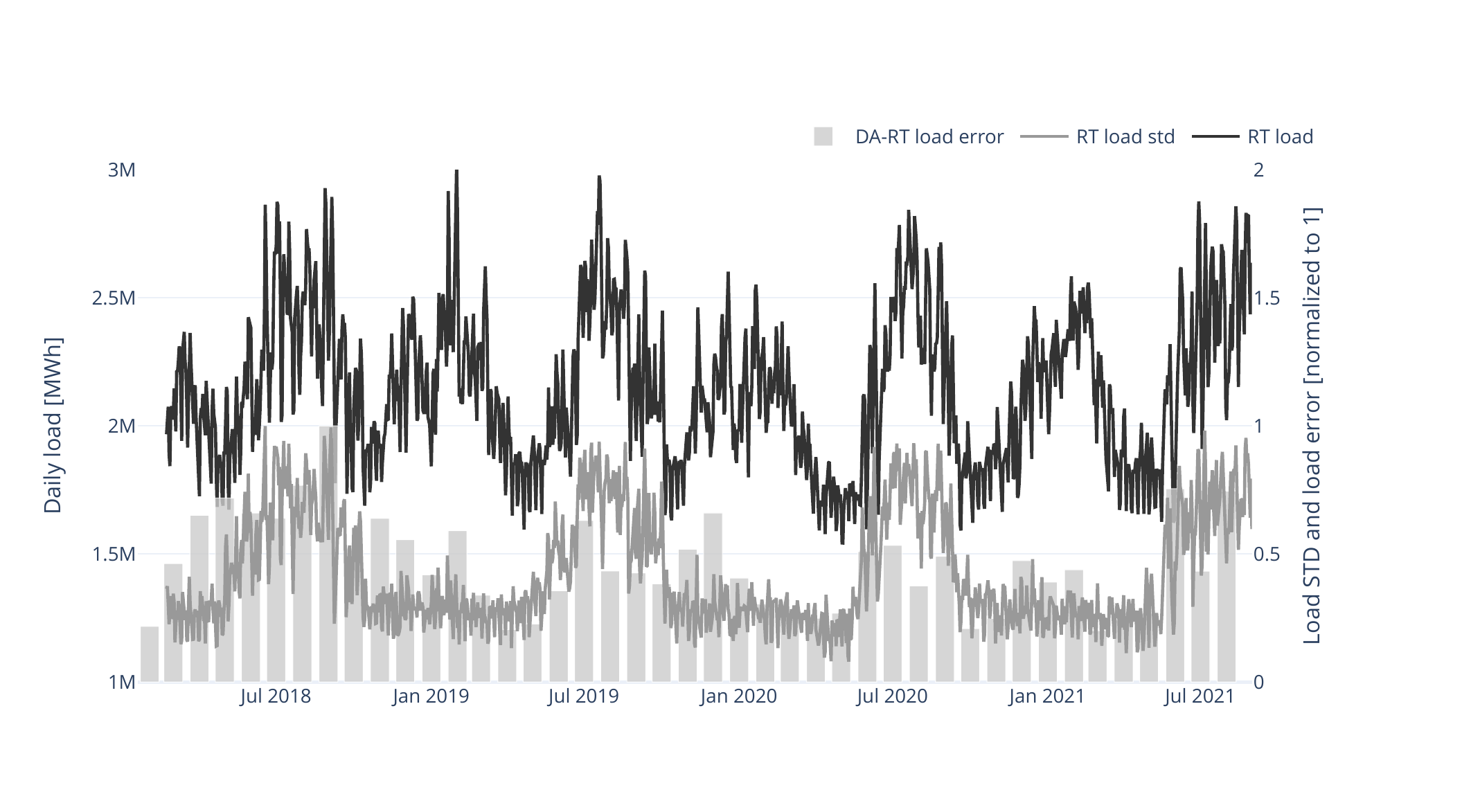}
\caption{Real-time daily load, daily load standard deviation and daily load error.}
\label{fig:PJM_load_val_std_error}
\end{figure}

\textsc{Other market controls---}Besides the drivers of side payments captured in the model of section \ref{sec:model}, the actual PJM market involves additional complexities that can affect the magnitude of side payments, such as the multi-period nature of the market or the presence of transmission grid constraints.
PJM provides a list of what they consider to be the main drivers of side payments which, besides generator market offers and load error between day-ahead and real-time, also includes (i) the load, (ii) the grid congestions, (iii) generation availability, (iv) the amount of self-scheduling as well as (v) emergency procedure events \citep{PJM2026drivers}.
These data are included as control variables $\textbf{X}$ in equation \eqref{eq:OLS_RT}.
We briefly describe these data below (more in depth description of the data is provided in Appendix \ref{sec:appendixPJMdata}).
(i) As far as the load is concerned, higher load is associated with more commitment of resources, which can give rises to more side payments. High load variation can also put the system under stress and lead to side payments. PJM load has a clear seasonality, with both a summer and a winter peak (cf. Figure \ref{fig:PJM_load_val_std_error}). Yet, the daily load profile exhibits significantly more variations during the summer peak rather than the winter peak.
(ii) Grid congestions also affect the prevalence of non-convexity and the magnitude of side payments: congestions tend to make the market more fragmented, which exacerbates the non-convexities of the market.
(iii) Unavailability of generation resources can also result into activations of costlier fast-start resources, at the right of the merit order curve, associated with more side payments.
(iv) The volume of self-scheduling impact the amount of flexibility available in the system, which also affect side payments. (v) Finally, emergency procedure events can indicate a high level of stress in the system.
Altogether, these elements can influence the side payments. For example, inspecting the data, the relatively low side payments observed from December 2019 to May 2020 on Figure \ref{fig:PJMuplifts} seem to be driven by favourable system conditions: a relatively mild winter (low load with low variability), fewer congestions, few emergency events, relatively low load error between DA and RT.

The control variables $\textbf{X}$ in the model of equation \eqref{eq:OLS_RT} are split in three categories which are \textit{market controls} (Henry hub gas price, day-ahead price and standard deviation, real-time load and standard deviation), \textit{date controls} (winter-summer and weekend-weekdays indicators) and \textit{system stress controls} (DA and RT self-scheduling, DA congestion, generation availability---both active units, generation capacity---and emergency events).

\paragraph*{Day-ahead model.} 
As noted earlier in this section, day-ahead side payments data exhibit two main features relative to real-time side payments (Table \ref{tab:upliftsData}): (i) they are substantially smaller in magnitude and (ii) they are equal to zero about $20 \%$ of the time.
Note that this share of days with zero day-ahead side payments is consistent with our model's prediction that virtual trading participation may eventually eliminate day-ahead side payments.
Instead of an OLS model, we estimate the following binary response logit model:
\begin{equation}
\mathbb{P} (SP_{DA, t} > 0 \vert x) = G (\alpha + \beta \text{Treatment}_t +  \eta \textbf{X}_t)
\end{equation} 
where $G$ is the logistic function ($G(z) = e^z/(1+e^z)$), $SP_{DA,t}$ corresponds to the day-ahead side payments at date $t$ and $\mathbb{P} (SP_{DA, t} > 0 \vert x)$ the probability of having positive side payments given the full set of explanatory variables $x$ (here, the Treatment$_t$ and the controls $\textbf{X}_t$). Treatment$_t$ is, as previously, the indicator variable representing the introduction of the transaction fee and $\textbf{X}_t$ are the control variables.
The main parameter of interest is $\beta$ which measures the treatment effect. 
We expect $\beta > 0$: there are higher chances of non-zero side payments when a transaction fee is introduced.

We use similar types of controls as for the real-time model, except that we obviously remove variables related to the real-time (we use DA load and not RT load, DA self-scheduling and not RT self-scheduling, etc.).
Similarly to the real-time model, the control variables $\textbf{X}$ are split in three categories which are market controls (Henry hub gas price, day-ahead load and standard deviation), date controls (winter-summer and weekend-weekdays indicators) and system stress controls (DA self-scheduling and generation availability---both active units, generation capacity).

\section{Main results} \label{sec:mainResults}
Tables \ref{tab:RTmainResults} and \ref{tab:DAmainResults} report the main results of both the real-time and day-ahead models. 
In both tables, column (I) reports the results of the full model, including all controls, for the period starting on February 20, 2018, to September 1, 2021. Columns (II) to (IV) correspond to models in which some controls are removed: columns (II) removes system stress controls, column (III) further removes date controls and column (IV) removes all controls. Columns (V) and (VI) are robustness tests on the time window. On February 20, 2018, the definition of UTC products changed. As we observe on Figure \ref{fig:PJM_virtual_bids_treatment}, this is followed by a period of roughly two months of steep increase of the UTC volume before it stabilizes, which likely corresponds to the time it took for traders to adapt their trading strategy. Column (V) filters these two first months (thus starting the model on April 20, 2018). On the other end of our time window, on September 1, 2021, PJM implemented a new pricing rule. Column (VI) adds four more months to the regression (thus terminating on December 31, 2021 instead of August 31, 2021).

As far as the real-time model is concerned, let us stress three main observations.
First, we find an average treatment effect of 135k\$. Thus our model estimates that the introduction of the transaction fee on virtual bids increased the daily real-time side payments by 135k\$ (with a 95\% confidence interval ranging from 103 to 168k\$). This represents an economically significant increase as, from Table \ref{tab:upliftsData}, the pre-treatment magnitude of the daily side payments was 167k\$. Thus the transaction fee roughly increased by 80\% the real-time side payments.
Second, from columns (I) to (VI), these findings are statistically significant and they are robust to sensitivities on both the control variables and on the time window.
Note that the slightly smaller effect in column (VI) is aligned with expectations since the implementation of the fast-start pricing policy on September 1st, 2021, is expected to have mitigated the side payments.
Third, both the FSR fixed costs and the load shocks positively affect the side payments. All else equal, higher FSR fixed costs---typically recovered through side payments---naturally lead to higher side payments.\footnote{The lower significance of FSR fixed cost compared to the other variables can be explained by the way we measure them. They are the average of the FSR fixed cost, whereas we expect only the least costly FSR units to have an impact. This is an artefact of our model where the fixed cost of FSR are linear, whereas it is of course not the case in reality.}
Similarly, higher load shocks, which are expected to trigger FSR activations, also lead to higher side payments.

\begin{table}[t]
\centering
\setlength{\tabcolsep}{4pt}
\begin{tabular}{lllllll}
  \toprule
 & (I) & (II) & (III) & (IV) & (V) & (VI) \\ 
  \hline
Treatment Effect & 135.09$^{***}$ & 159.79$^{***}$ & 160.67$^{***}$ & 207.78$^{***}$ & 127.44$^{***}$ & 108.75$^{***}$ \\ 
 (in K dollars)   & (16.55) & (16.24) & (16.44) & (17.93) & (17.13) & (16.77) \\ 
        &   &   &   &   &   &   \\ 
FSR fixed cost & 33.97$^{**}$ & 11.4 & 6.51 & 31.59$^{***}$ & 42.38$^{**}$ & 13.3 \\ 
      & (12.64) & (8.63) & (8.71) & (9.57) & (13.29) & (10.49) \\ 
          &   &   &   &   &   &   \\ 
DA-RT load error & 0.0003$^{***}$ & 0.0003$^{***}$ & 0.0003$^{***}$ & 0.0005$^{***}$ & 0.0003$^{***}$ & 0.0003$^{***}$ \\ 
       & (3.7e-05) & (3.7e-05) & (3.7e-05) & (4e-05) & (3.9e-05) & (3.7e-05) \\ 
                &   &   &   &   &   &   \\ 
  System stress controls & X &   &   &   & X & X \\ 
  Date controls & X & X &   &   & X & X \\ 
  Market controls  & X & X & X &   & X & X \\ 
              &   &   &   &   &   &   \\ 
  Sample size & 1276 & 1276 & 1276 & 1276 & 1216 & 1395 \\ 
  R squared & 0.46 & 0.42 & 0.41 & 0.21 & 0.46 & 0.41 \\ 
   \bottomrule
\end{tabular}
\caption{Real-time side payments results (OLS model).} 
\label{tab:RTmainResults}
\end{table}

Regarding the day-ahead side payments, our model estimates a logit coefficient of 0.7. Thus, as expected, the introduction of the transaction fee increased the probability of having non-zero side payments in day-ahead.
As for the real-time model, we find this estimate is statistically significant and robust to our sensitivities on both the control variables and on the time window.
A logit coefficient of 0.7 means the treatment double the odds ($e^{0.695} \approx 2.0$) of having non-zero side payments in day-ahead.
Given an average probability of non-zero day-ahead side payments before treatment was 76\%, the logit coefficient means the treatment increased on average by $+10$\% the chances of having non-zero side payments in day-ahead.\footnote{$p_1 = 0.76 = \frac{e^z}{1+e^z}$ from which we compute $e^z = 3.16$. Then $p_2 = \frac{e^z e^{\beta}}{1+e^z e^{\beta}}$ with $\beta = 0.695$ from which we compute $p_2 = 0.86$ thus $p_2 - p_1 = 0.1$.}

\begin{table}[t]
\centering
\setlength{\tabcolsep}{4pt}
\begin{tabular}{lllllll}
  \toprule
 & (I) & (II) & (III) & (IV) & (V) & (VI) \\ 
  \hline
Logit Coefficient & 0.6954$^{***}$ & 0.7156$^{***}$ & 0.5416$^{**}$ & 0.5189$^{**}$ & 0.7840$^{***}$ & 0.7232$^{***}$ \\ 
    & (0.1996) & (0.1923) & (0.1870) & (0.1757) & (0.2005) & (0.1880) \\ 
     &   &   &   &   &   &   \\ 
  System stress controls & X &   &   &   & X & X \\ 
  Date controls & X & X &   &   & X & X \\ 
  Market controls & X & X & X &   & X & X \\ 
      &   &   &   &   &   &   \\ 
  Sample size & 1276 & 1276 & 1276 & 1276 & 1216 & 1395 \\ 
   \bottomrule
\end{tabular}
\caption{Day-Ahead side payments results (logit model).} 
\label{tab:DAmainResults}
\end{table}

\section{Discussion and conclusions} \label{sec:ccl}

The effect of financial trading in a non-convex electricity market has been subject to debate among stakeholders and in the scientific literature.
In a note from 2015, PJM argued that virtual bids could in principle give rise to side payments by distorting the commitments made in day-ahead, leading to under or over commitment, associated with side payments \citep{PJM2015}. PJM's argument was not based on an equilibrium analysis though, but on stylized examples that assume exogenously some bidding behaviour for virtual traders (while, as we argue in our model, what matters is what happens \textit{at the equilibrium} and \textit{in expectation}). From this analysis, PJM recommended charging the virtual bids a transaction fee to compensate for the side payments they may cause.

In 2016, Hogan argued against the idea of allocating uplift charges to virtual traders as this increase in transaction cost would deter participation and have material consequences on the efficiency of the market \citep{hogan2016virtual}.

Later in 2020, the FERC order that led to the implementation of the transaction fee on UTCs (which we analyzed in section \ref{sec:empiricalStrat}) gave rise to contrasted viewpoints on the benefits of virtual trading. 
The utilities were in favor of this policy: ``PJM Utilities Coalition conclude that because UTCs are transacting in the markets, they should receive their respective share of the costs that result from market operations.'' \citep{FERC2020}
In contrast, some trading firms active in PJM market argued that ``because UTCs promote price convergence and help with price formation, allocation of uplift to UTCs disincentivizes UTCs' efficiency enhancing characteristics on high priced days when there is risk of high uplift.'' \citep{FERC2020}

Our model and empirical analysis contribute to shed light on this debate.
Our article argues that the introduction of financial trading in an electricity market induces a \textit{smoothing effect}. 
Financial traders are arbitragers that bid close to the market price, resulting in the addition of a large volume of convex bids close to the market margin.
Our model shows that, by internalizing the fixed costs in their bids and anticipating the real-time system conditions, the financial traders mitigate the side payments needed to sustain an equilibrium.
In day-ahead, virtual traders set the prices and as a consequence, fixed costs are reflected in the uniform price signal instead of being refunded by side payments. This eliminates day-ahead side payments (in the scope of our model).
In real-time, higher price convergence and better commitment decisions in day-ahead lead to less lumpy activations of fast-start resources in real-time usually associated with side payments.  
Our empirical analysis confirm the predictions of our model. According to our analysis, the introduction of a transaction fee in PJM market on the financial traders, on November 1, 2020, led to a substantial drop of financial trading volume which in turn led to an increase of side payments. In day-ahead, the likelihood of non-zero side payments increased by 10\% following the introduction of the transaction fee. In real-time, the amount of side payments increased by 80\%.

FERC justified the introduction of transaction fees on virtual UTC bids on the ground that other virtual bids (INC and DEC) were already subject to similar fees, which was judged ``unjust, unreasonable, and unduly preferential''. A natural alternative proposition would have been to remove transaction fees on INC and DEC, instead of introducing a fee on UTC. From our results, virtual trading reduces side payments and a transaction fee is counter-productive. As stressed by \cite{hogan2016virtual}, the inexistence of an equilibrium and associated side payments have their source in non-convexities, not in financial trading.

\bibliographystyle{plainnat}
\bibliography{references.bib}

\newpage
\appendix

\renewcommand{\thefigure}{\Alph{section}.\arabic{figure}}
\renewcommand{\thetable}{\Alph{section}.\arabic{table}}
\renewcommand{\theequation}{\Alph{section}.\arabic{equation}}
\setcounter{equation}{0}
\setcounter{figure}{0}
\setcounter{table}{0}

\section{Proofs of the propositions}
\begin{proof}[Proof of Proposition \ref{prop:comparativeStaticsK}]
We prove the four points. Let us first define for convenience $A = c_F^2 + (2 c_S \overline{D})^2 + 2c_F c_S \overline{D} + 4 c_S s_S \overline{D}$, the expression inside the square root in \eqref{eq:k*}.
\begin{enumerate}
\item $\partial k^*/\partial s_S < 0$ is straightforward from \eqref{eq:k*}. 
\item $\partial k^*/\partial c_F > 0$ can be seen from equation \eqref{eq:2orderk*}: this is a convex second-order polynomial and a higher $c_F$ shifts the function upward ($\forall k \in [0, \overline{D}]$) thus the smaller root ($k^*$) moves to the right. Alternatively, using \eqref{eq:k*}, one can compute:
\begin{align*}
\frac{\partial k^*}{\partial c_F} & = \frac{1}{c_S} \left( 1 - \frac{c_F + c_S \overline{D}}{\sqrt{A}}  \right) > 0
\end{align*}
Indeed, $\sqrt{A} = \sqrt{(c_F + c_S \overline{D})^2 + 3 (c_S \overline{D})^2 + 4 c_S s_S \overline{D}} > c_F + c_S \overline{D}$
\item $\partial k^*/\partial c_S < 0$ can be seen from equation \eqref{eq:2orderk*}: clearly, $\forall k \in [0, \overline{D}]$, a higher $c_S$ shifts the function downward thus the smallest root ($k^*$) moves to the left.
Alternatively, using \eqref{eq:k*}, one can compute:
\begin{align*}
\frac{\partial k^*}{\partial c_S}  & =  \frac{-c_F}{c_S^2} + \frac{\sqrt{A}}{c_S^2} - \frac{4 c_S \overline{D}^2 + + c_F \overline{D} + 2 s_S \overline{D}}{c_S \sqrt{A}}  < 0 \\
& \Leftrightarrow -c_F \sqrt{A} + c_F^2 + c_F c_S \overline{D} + 2 c_S s_S \overline{D} < 0 \\
& \Leftrightarrow \sqrt{c_F^2 + (2 c_S \overline{D})^2 + 2c_F c_S \overline{D} + 4 c_S s_S \overline{D}} > c_F + c_S \overline{D} + \frac{2 c_S s_S \overline{D}}{c_F} \\
& \Leftrightarrow \frac{s_S^2}{c_F^2} + \frac{s_S}{c_F} < \frac{3}{4} \\
& \Leftrightarrow c_F > 2 s_S 
\end{align*}
which is nothing more than the condition for $k^* > 0$.
\item $\partial k^*/\partial s_F = 0$ is straightforward from \eqref{eq:k*}. 
\end{enumerate}
\end{proof}

\begin{proof}[Proof of Proposition \ref{prop:sidePayments}]
The day-ahead side payments drop from $s_S d$ to zero. The real-time side payments are reduced by $\frac{s_F k}{2} \left(1 - \frac{k}{2 \overline{D}} \right)$ (which is $>0$ for any $k<\overline{D}$).
\end{proof}

\begin{proof}[Proof of Proposition \ref{prop:sidePayments2}]
$k^*$ does not depend on $s_F$, thus $\mathbb{E}_{\epsilon} (SP_{RT,\epsilon})$ is a linear function in $s_F$. We then compute the derivative:
\begin{align*}
\frac{\partial \mathbb{E}_{\epsilon} (SP_{RT,\epsilon})}{\partial s_F} & = \frac{1}{2 \overline{D}} \int_k^{\overline{D}} (\epsilon-k) d\epsilon > 0
\end{align*}
\end{proof}

\begin{proof}[Proof of Proposition \ref{prop:efficiency1}]
We prove that
\begin{equation} 
s_S - \frac{k^2 c_S}{6 \overline{D}} + \frac{k(c_F+s_F)}{4 \overline{D}} + \frac{k c_S}{4} - \frac{(c_F + s_F)}{2} <0 \label{eq:costReduct}
\end{equation}
Injecting \eqref{eq:2orderk*} into \eqref{eq:costReduct} to replace $k^2$, and rearranging the terms, the left-hand-side leads to
$$L(k) = k \left( \frac{-c_F}{12 \overline{D}} + \frac{s_F}{4 \overline{D}} - \frac{5c_S}{12} \right) + \frac{s_S}{3} - \frac{s_F}{2} - \frac{c_F}{6} $$
$L(k)$ is a linear function in $k$, with $k^* \in [0, \overline{D}]$. A sufficient condition for $L(k^*)<0$ is thus $L(0)<0$ and $L(\overline{D})<0$. Using the model assumptions $c_S \overline{D} + s_S < c_F$ and $s_F > s_S$, we conclude:
\begin{align*}
L(0) & = \frac{s_S}{3} - \frac{s_F}{2} - \frac{c_F}{6}  < 0 \\
L(\overline{D}) &= - \frac{c_F}{4} -  \frac{s_F}{4}-  \frac{5 c_S \overline{D}}{12} +  \frac{s_S}{3} <0
\end{align*}
which proves our statement. (Alternatively, one can prove Proposition \ref{prop:efficiency1} by analysing the shape, the first and second derivatives and the optimum $k^{**}$, of the cost curve $Cost(k)$, as done in the proof of Proposition \ref{prop:efficiency2}.)
\end{proof}

\begin{proof}[Proof of Proposition \ref{prop:efficiency2}]
The first best optimal commitment in day-ahead can be computed by optimizing the cost function $Cost(k)$ in equation \eqref{eq:cost_withVB} with respect to $k$.
Computing the first and second derivatives lead to
\begin{align}
\frac{\partial Cost}{\partial k} & = s_S - \frac{c_S k^2}{2 \overline{D}} + k \left( \frac{c_F+s_F}{2 \overline{D}} + \frac{c_S}{2} \right) - \frac{c_F + s_F}{2}  \label{eq:optimalCostDerivative} \\
\frac{\partial^2 Cost}{\partial k^2} & = -\frac{c_S k}{\overline{D}} + \frac{c_F + s_F}{2\overline{D}} + \frac{c_S}{2}
\end{align}
We can check that $\frac{\partial Cost}{\partial k} \vert_0 < 0$ (as $c_F + s_F > 2 s_S$) and $\frac{\partial Cost}{\partial k} \vert_{\overline{D}} > 0$ ; as well as $\frac{\partial^2 Cost}{\partial k^2} >0$ for $k \in [0, \overline{D}]$. Thus $Cost(k)$ is a polynomial cubic function in $k$, which is convex for $k \in [0, \overline{D}]$, decreasing at $k=0$ and increasing at $k=\overline{D}$. This is illustrated on Figure \ref{fig:costCurvek}.
The optimal $k^{**}$ satisfies $\frac{\partial Cost}{\partial k} = 0$, thus 
\begin{equation} \label{eq:optimalK}
c_S k^2 + k(-c_S \overline{D} - c_F - s_F) + (c_F + s_F - 2 s_S)\overline{D} = 0
\end{equation}
which differs from \eqref{eq:2orderk*}.
Clearly, $k^{**} > k^*$ as the second-order function defined by \eqref{eq:optimalK} is above the second-order function defined by \eqref{eq:2orderk*} ($k^2$ coefficient is bigger and $\overline{D} s_F \geq k s_F$ for $k \in [0, \overline{D}]$). This implies the smallest root of \eqref{eq:2orderk*} is smaller than \eqref{eq:optimalK}.
\end{proof}

\begin{proof}[Proof of Proposition \ref{prop:transactionCost}]
Let us prove the inequalities stated:
\begin{itemize}
\item $k^t \leq k^*$ is immediate from equation \eqref{eq:kcondition_with_t}: the second order polynomial defined in \eqref{eq:kcondition_with_t} is below from the one defined in \eqref{eq:2orderk*}, thus the smallest root is smaller.
\item $Cost(k^*) \leq Cost(k^t) \leq Cost(0)$ is immediate from the expression of the cost function $Cost(k)$ which is decreasing at $k=0$, convex for $k \in [0, \overline{D}]$ and minimized at $k^{**}$ with $k^t \leq k^* \leq k^{**}$ (cf. Proposition \ref{prop:efficiency2}).
\item $SP(k^t) \leq SP(0)$ is immediate from the expression of the side payments in \eqref{eq:VBequilSPRT}. 
\item $SP(k^*) \leq SP(k^t) \leq SP(0)$ can be seen by computing the derivative of the side payments: $\frac{\partial \mathbb{E}_{\epsilon} (SP_{RT,\epsilon})}{\partial k} = \frac{s_F}{2} \left( \frac{k}{\overline{D}} - 1 \right) < 0$ for $k\in [0, \overline{D}]$. Thus the side payments are a decreasing function of $k$ for $k\in [0, \overline{D}]$. The conclusion follows from $0 \leq k^t \leq k^* < \overline{D}$.
\end{itemize}
\end{proof}

\section{Some more numerical examples} \label{sec:appendixMoreExamples}
We construct a portfolio of power plants, introducing some heterogeneity between the SSR and the FSR. We build on the portfolio from the example in section \ref{sec:model}, perturbing randomly both the costs, the minimum and the maximum production limits of both SSR and FSR. The resulting merit order curve is shown on Figure \ref{fig:example2_MO}. This resembles the merit order curve in Figure \ref{fig:example_MO}, but with more heterogeneity between the power plants.
One can see from Figure \ref{fig:example2_results} that all the previous observations from section \ref{sec:model} (both the model and the simulations on the numerical example) still apply to this example.

Notice that, because of the non-convexity in the game, a pure strategy equilibrium is not guaranteed to exist. Instead, the game may admit a mixed-strategy equilibrium, in which virtual traders randomize, for example, between the two strategies that yield zero expected profit.

\begin{figure}[t!]
    \centering
    \begin{subfigure}[b]{0.5\textwidth}
        \centering
        \includegraphics[width=\textwidth]{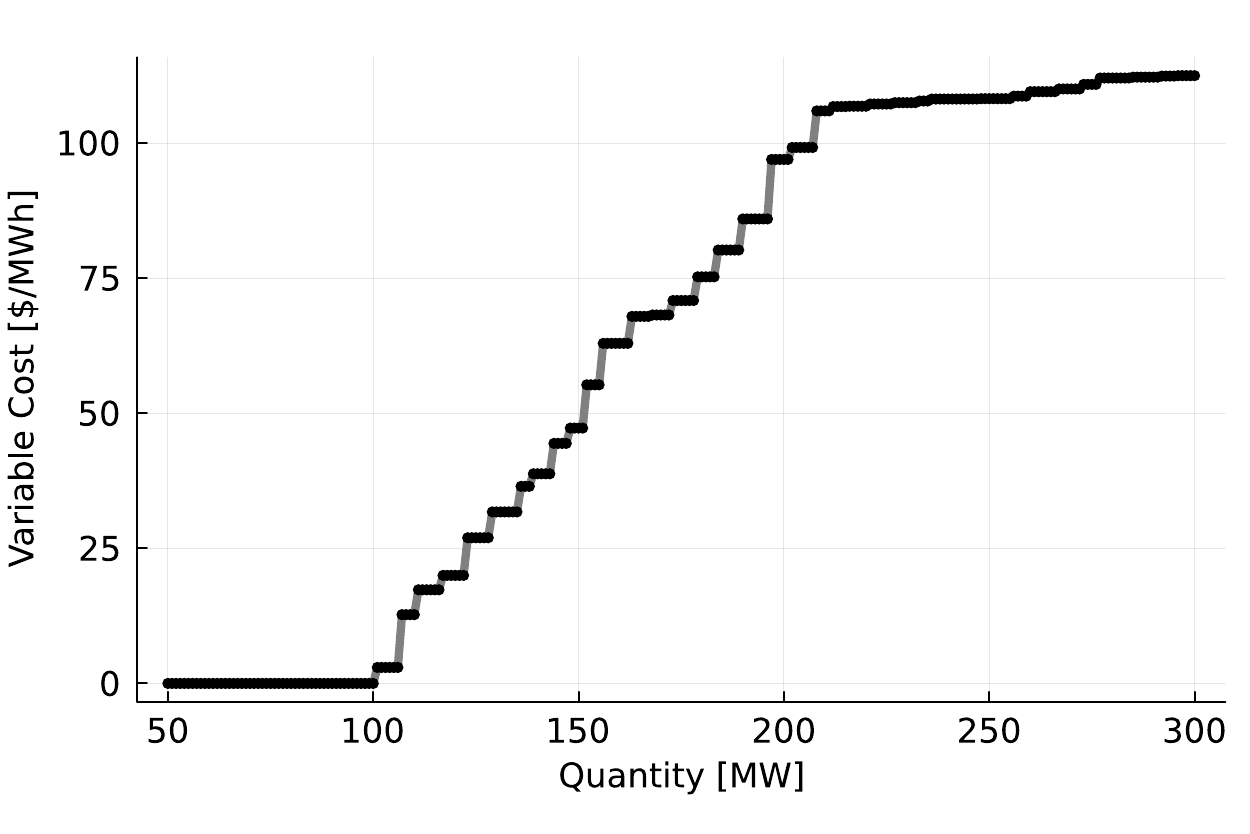}
        \caption{Merit order curve}
        \label{fig:example2_MO}
    \end{subfigure}%
    ~ 
    \begin{subfigure}[b]{0.5\textwidth}
        \centering
        \includegraphics[width=\textwidth]{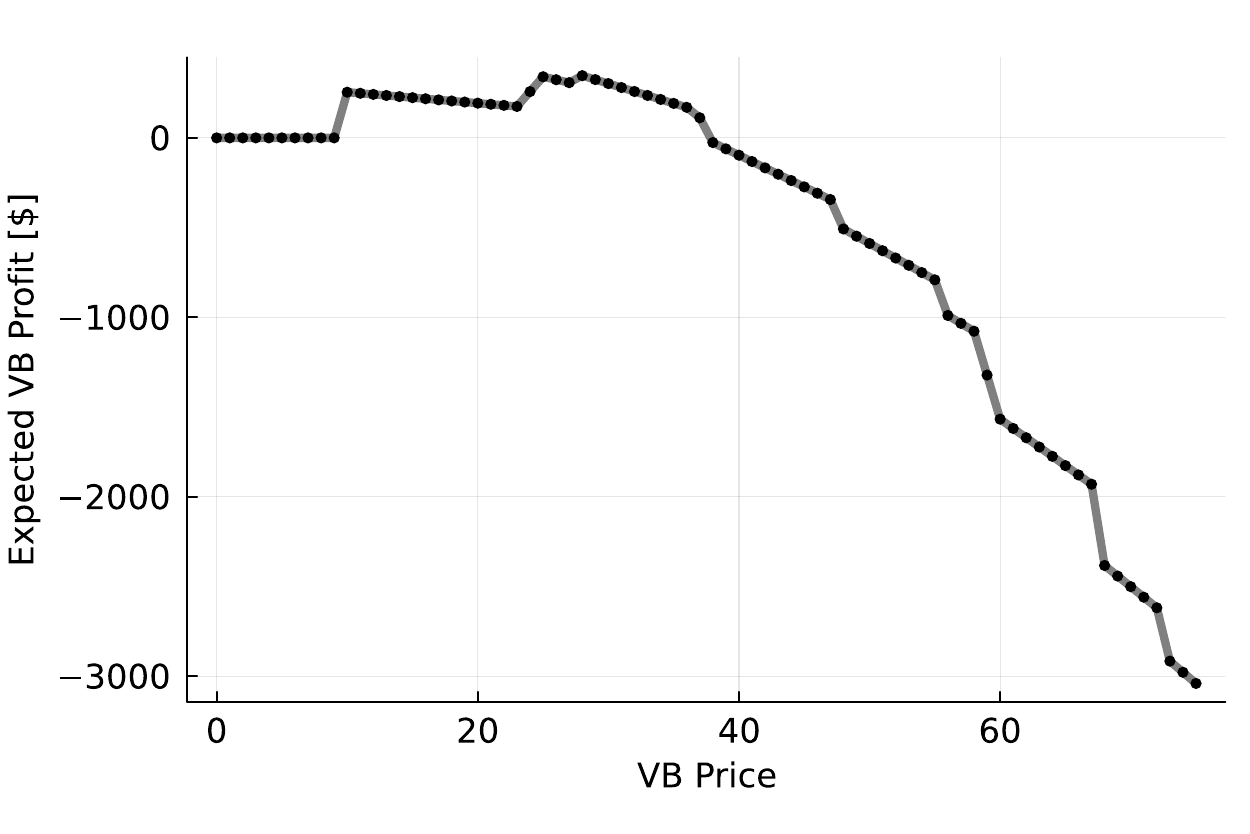}
        \caption{Virtual bidders profit}
        \label{fig:example2_VBprofit}
    \end{subfigure}
    \begin{subfigure}[b]{0.5\textwidth}
        \centering
        \includegraphics[width=\textwidth]{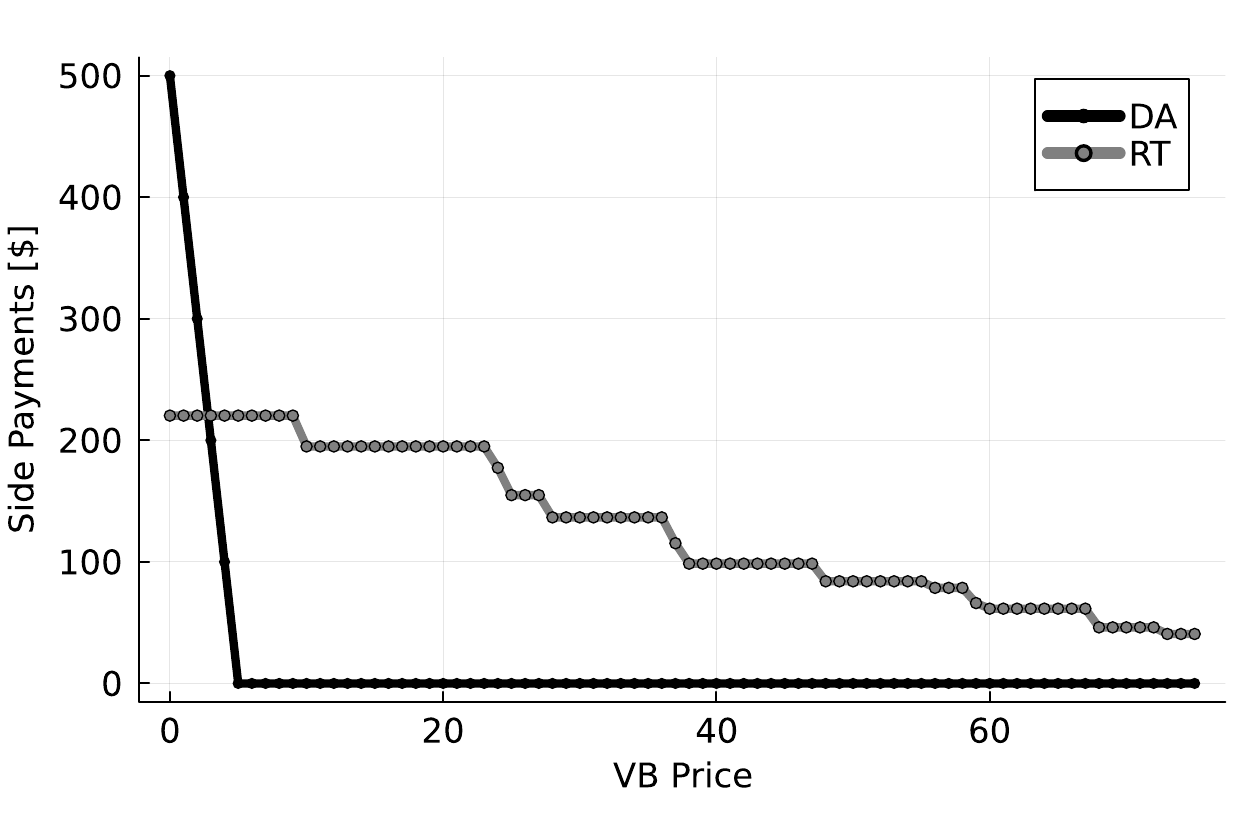}
        \caption{Side Payments}
        \label{fig:example2_SP}
    \end{subfigure}%
    ~ 
    \begin{subfigure}[b]{0.5\textwidth}
        \centering
        \includegraphics[width=\textwidth]{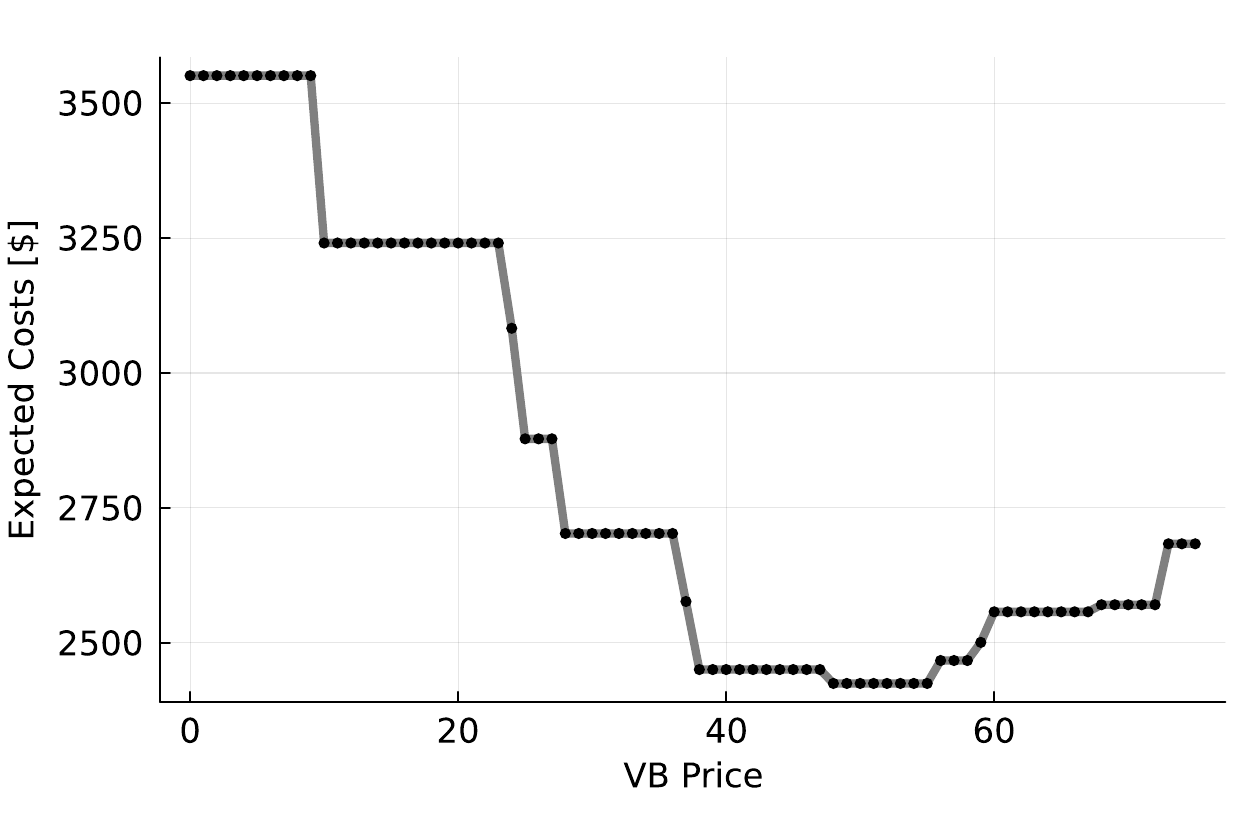}
        \caption{Total expected costs}
        \label{fig:example2_Costs}
    \end{subfigure}
    \caption{Results of the perturbed discrete example.}
    \label{fig:example2_results}
\end{figure}

\section{PJM data} \label{sec:appendixPJMdata}
This appendix provides more in-depth information about PJM data. We comment about each data one by one.

\textsc{Missing data}. As mentioned in section \ref{sec:empiricalStrat}, our database includes 1276 dates, excluding 12 dates with missing data. The dates with missing data are: 2018-12-25, 2018-12-29, 2018-12-31, 2019-10-09, 2019-12-31, 2020-07-25, 2020-12-29, 2020-12-31, 2021-01-07, 2021-01-24, 2021-07-25, 2021-05-27. 
The missing data are coming from side payments data, self-schedule data and DA load data. Side payments data are missing for 2018-12-25 and 2018-12-29. Self-schedule data are missing for 2018-12-31, 2019-10-09, 2019-12-31, 2020-07-25, 2020-12-31, 2021-01-07, 2021-05-27, 2021-11-01 and 2021-12-13. DA load data are missing for 2020-12-29, 2021-01-24 and 2021-07-25.

\textsc{Side payments data}. Figure \ref{fig:PJMupliftsDistribution} shows the distribution of the side payments (related to non-convexities).
Figure \ref{fig:PJMupliftsAllCat} shows all the types of uplift payments that exist in PJM, divided per category.
Other uplifts payments can be divided into energy market uplifts, ancillary services uplifts and reliability services uplifts.
Energy market uplift---besides Balancing Operating Reserve---include ``Balancing Operating Reserve Lost Opportunity Cost'' which are uplift due to out-of-merit re-dispatch for transmission constraint control. Ancillary services uplifts include ``Regulation and Frequency Response Service'', ``Synchronized Reserve'' and ``Non-Synchronized Reserve''. Reliability services uplifts include ``Black Start Service'' and ``Reactive Supply and Voltage Control Service''.

\begin{figure}
\centering
\includegraphics[width=0.8\textwidth]{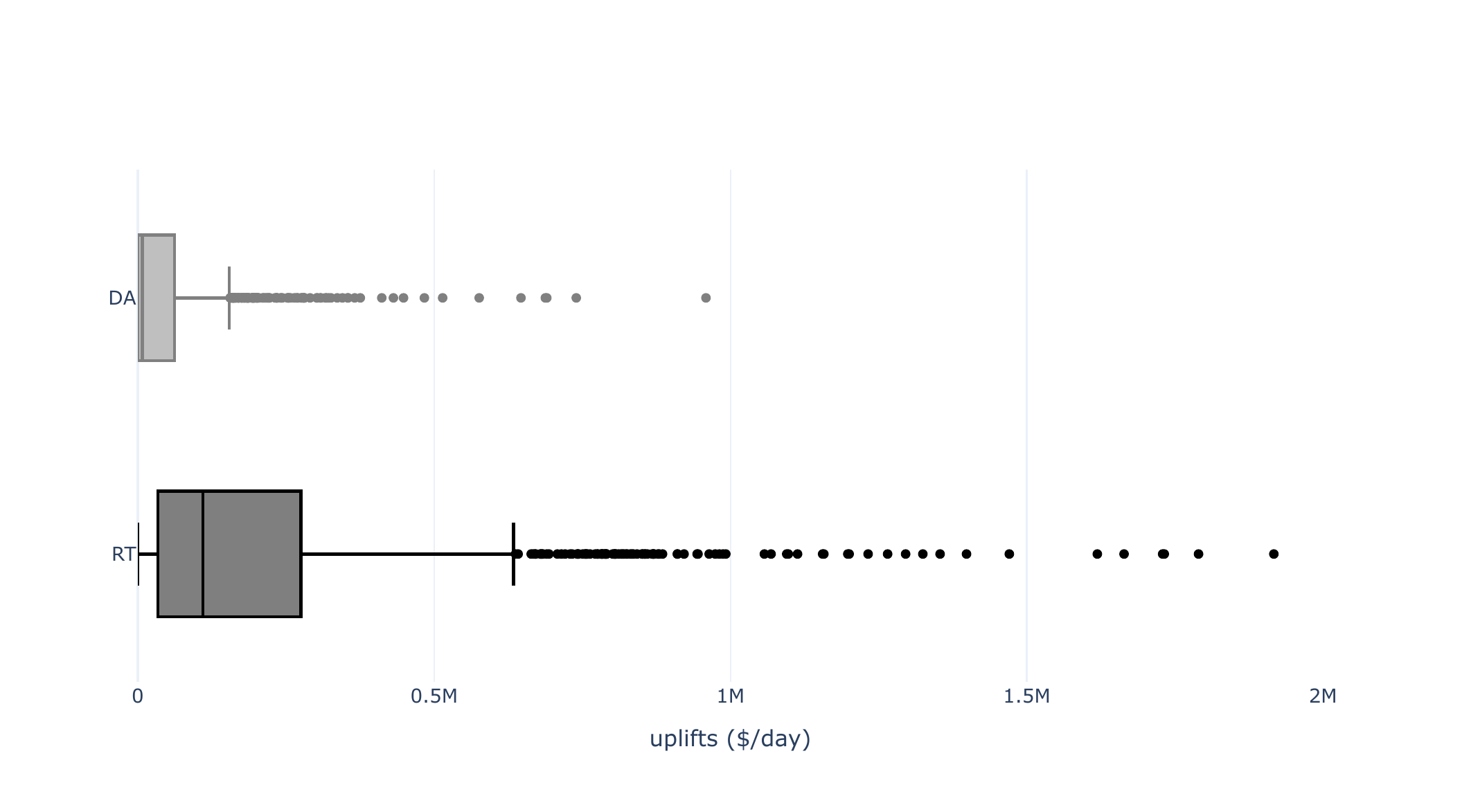}
\caption{PJM daily side payments distribution.}
\label{fig:PJMupliftsDistribution}
\end{figure}

\begin{figure}
\centering
\includegraphics[width=0.8\textwidth]{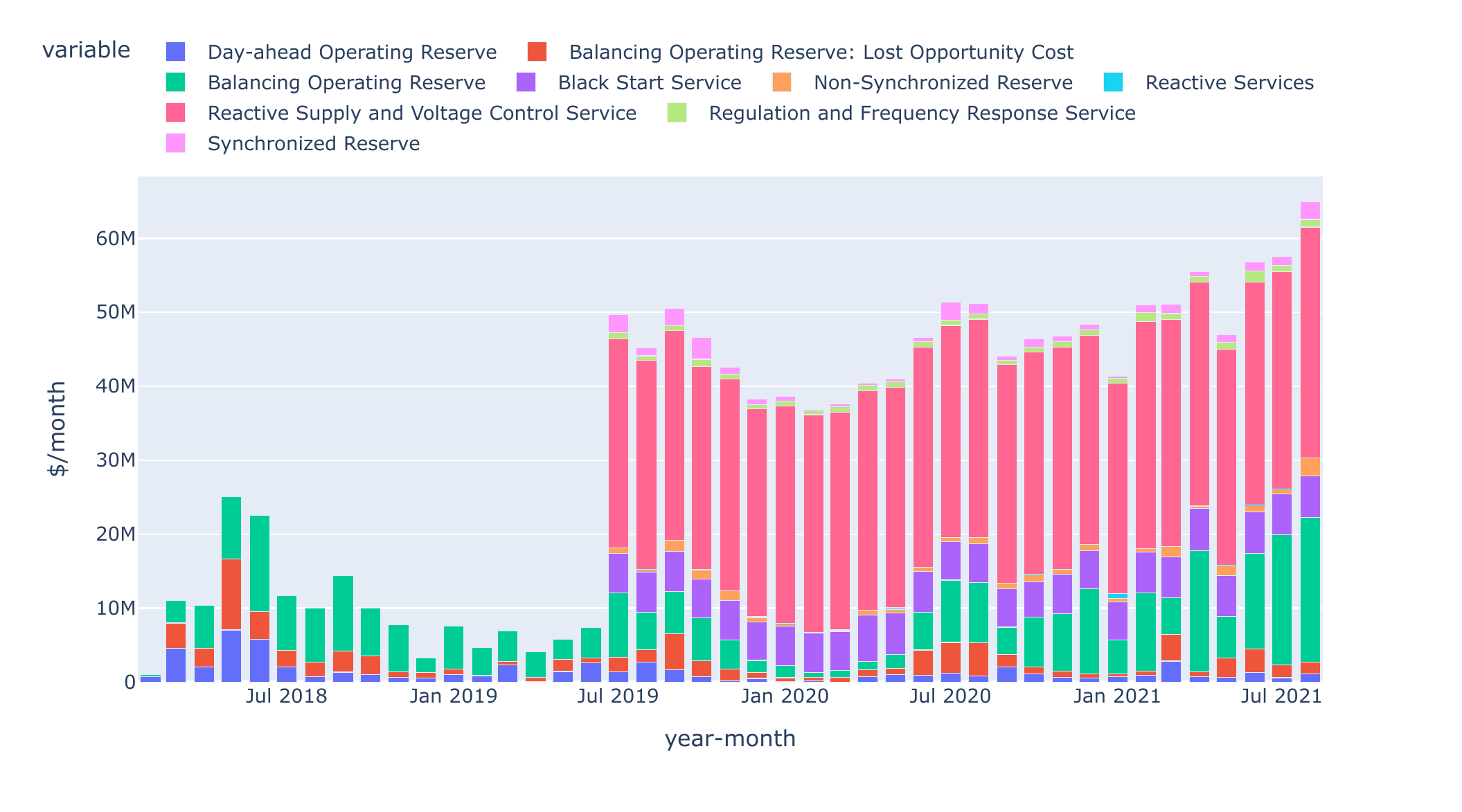}
\caption{Overview of all categories of PJM ``uplifts''.}
\label{fig:PJMupliftsAllCat}
\end{figure}

\textsc{Generation bids}. Figure \ref{fig:FSRdata} provides the distribution of the capacity and fixed costs of fast-start resources. As noted in section \ref{sec:empiricalStrat}, fast start resources are mostly small power plants, with a capacity between 0 and 100MW and a median at 43MW.
Their start-up costs are mostly between 0 and 10,000\$. 

\begin{figure}[t]
    \centering
    \begin{subfigure}[b]{0.5\textwidth}
        \centering
        \includegraphics[width=\textwidth]{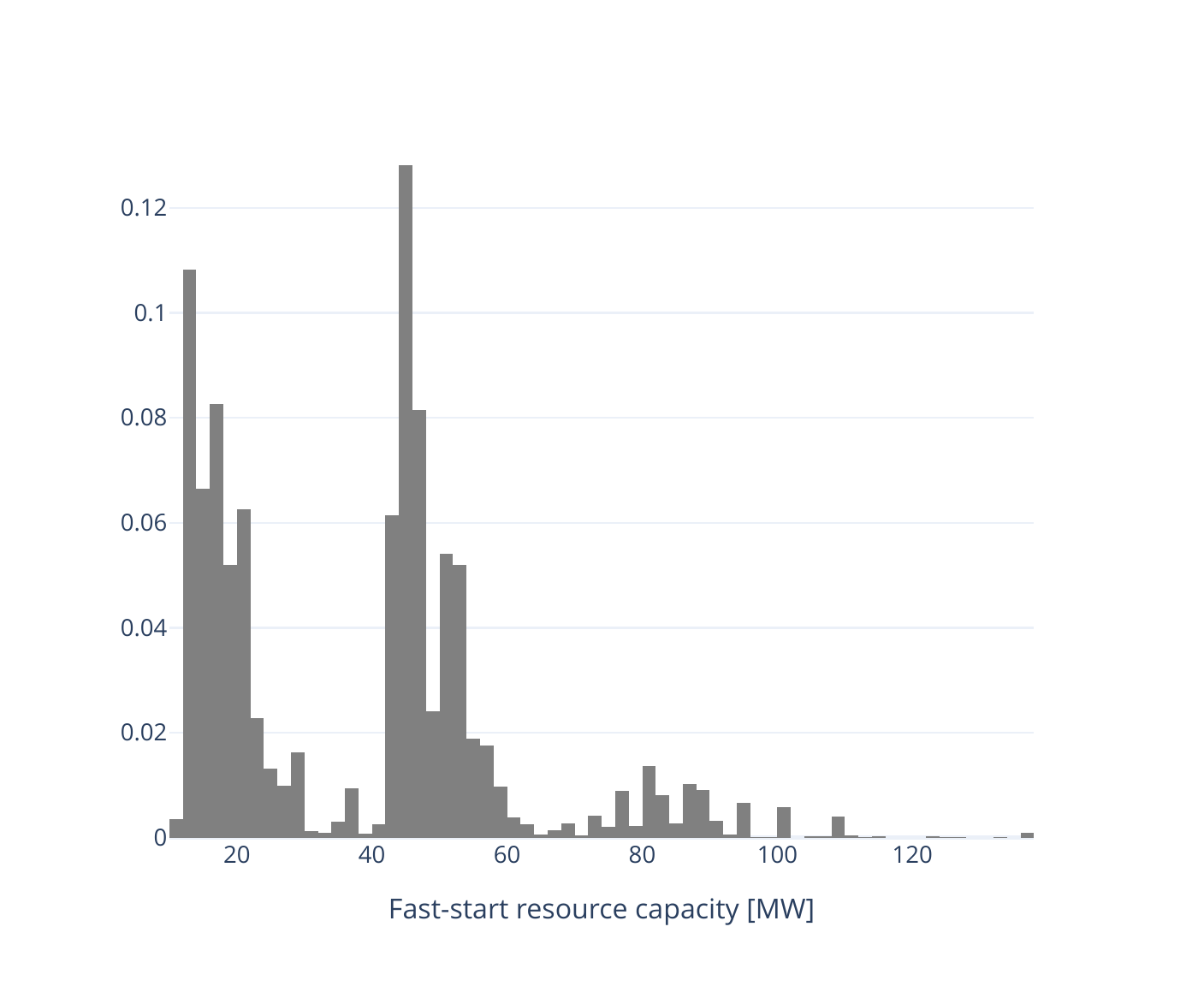}
        \caption{FSR capacity distribution}
        \label{fig:FSRdataCapaDist}
    \end{subfigure}%
    ~ 
    \begin{subfigure}[b]{0.5\textwidth}
        \centering
        \includegraphics[width=\textwidth]{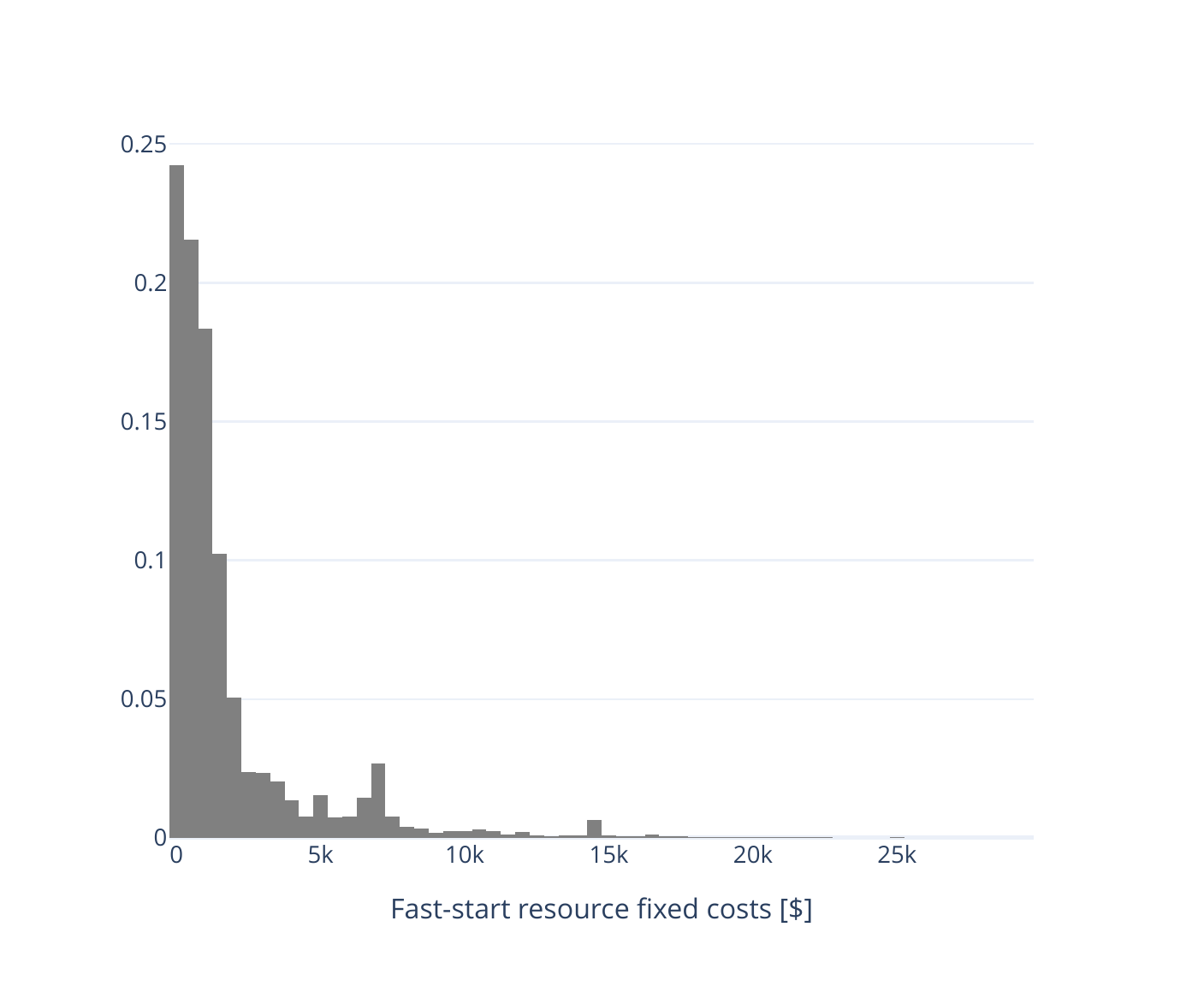}
        \caption{FSR fixed costs distribution}
        \label{fig:FSRdataCostDist}
    \end{subfigure}
    \caption{PJM fast-start resource capacity and fixed costs distributions.}
    \label{fig:FSRdata}
\end{figure}

\textsc{Load}. We aggregate PJM hourly load into daily load values by summation. We also compute the standard deviation of the hourly load within each day as a measure of load variations: a high standard deviation indicates a sharp difference between peak and offpeak loads within a day, often associated with steep load increase. 
Both measures are computed for day-ahead and real-time loads.
Figure \ref{fig:PJM_load_val_std_error} shows the resulting time series for the real-time loads.

\textsc{Grid congestions}. Direct data on grid congestions, such as the shadow prices associated to each line or the congestion rent, are not available on PJM. Yet, the severity of congestions can be inferred from prices data. We test two metrics of congestion, that are included in the controls. One is based on a load-weighted average of the congestion component of PJM total price. The other is based on the standard deviation of the locational prices at each hour. Days with high standard deviation of locational prices are indeed likely to be days where grid constraints are binding.

\textsc{Market prices}. PJM provides times series of hourly locational price of electricity in day-ahead and real-time. We compute aggregate daily prices (cf. Figure \ref{fig:PJM_prices_DARTgas}). We use the system price (price at the reference hub) as a control in our regression.

\begin{figure}
\centering
\includegraphics[width=0.8\textwidth]{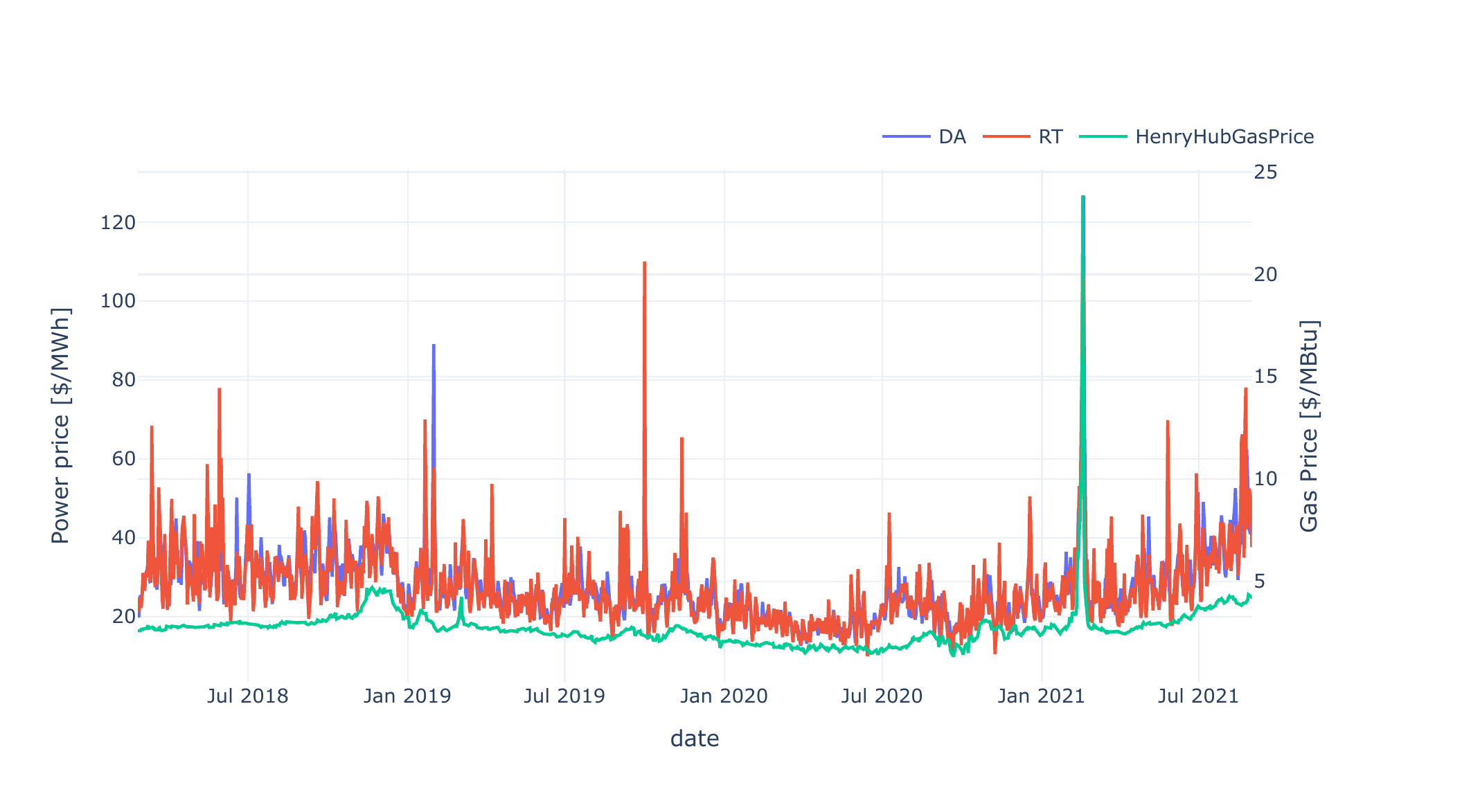}
\caption{Gas price and PJM prices.}
\label{fig:PJM_prices_DARTgas}
\end{figure}

\textsc{Generation availability}. PJM bids data also enable to compute the generation capacity available each day (i.e. some plants might be unavailable due to planned outages or for other reasons). The resulting time series of both number of plants available and MW capacity available is provided in Figure \ref{fig:PJM_genCapaAndActiveUnits}. There is a clear seasonality in plants availability, with most non-availabilities occurring during the fall and the spring, thus outside the winter and summer peak loads.

\begin{figure}
\centering
\includegraphics[width=0.8\textwidth]{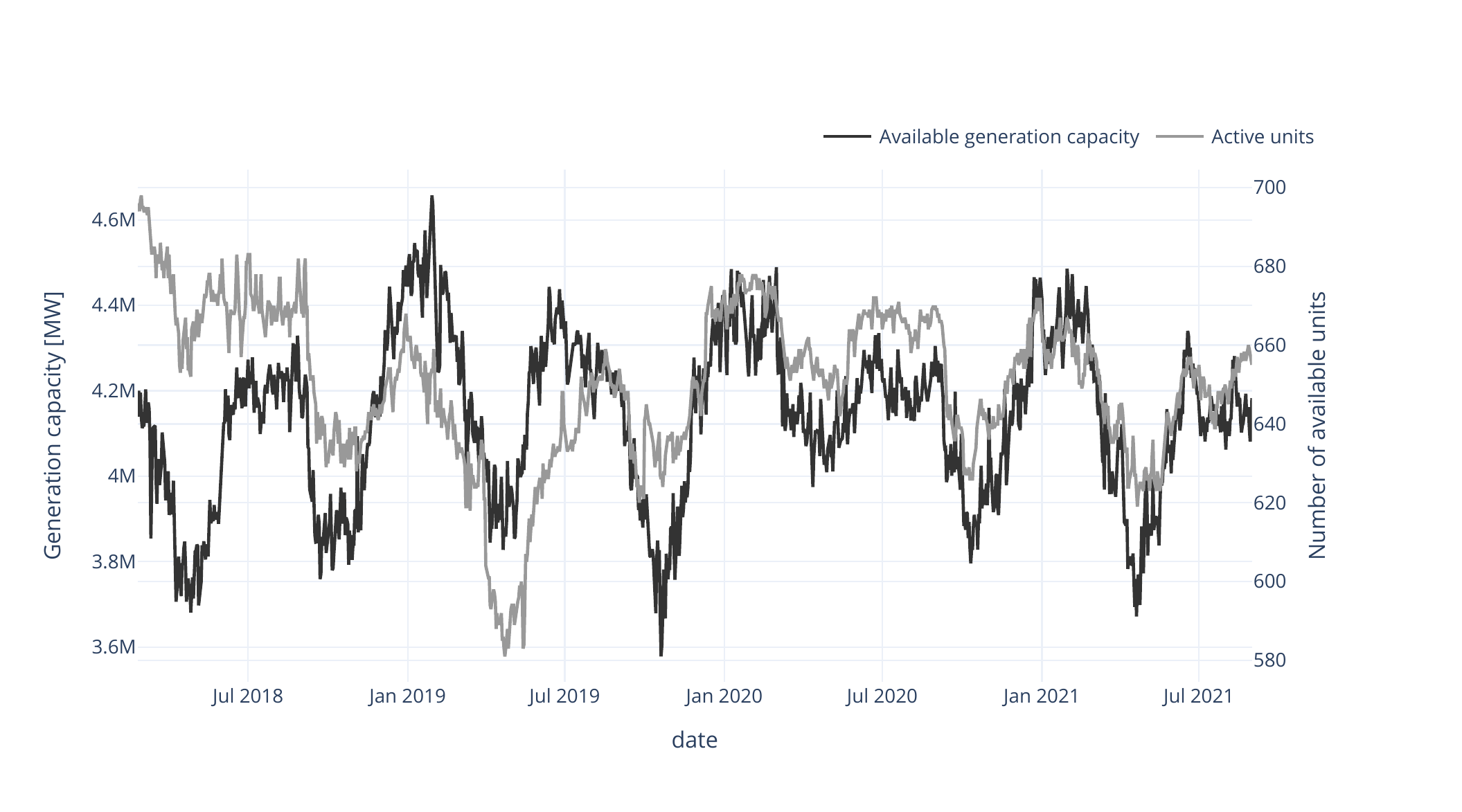}
\caption{PJM daily available generation capacity and number of active units.}
\label{fig:PJM_genCapaAndActiveUnits}
\end{figure}

\begin{figure}
\centering
\includegraphics[width=0.8\textwidth]{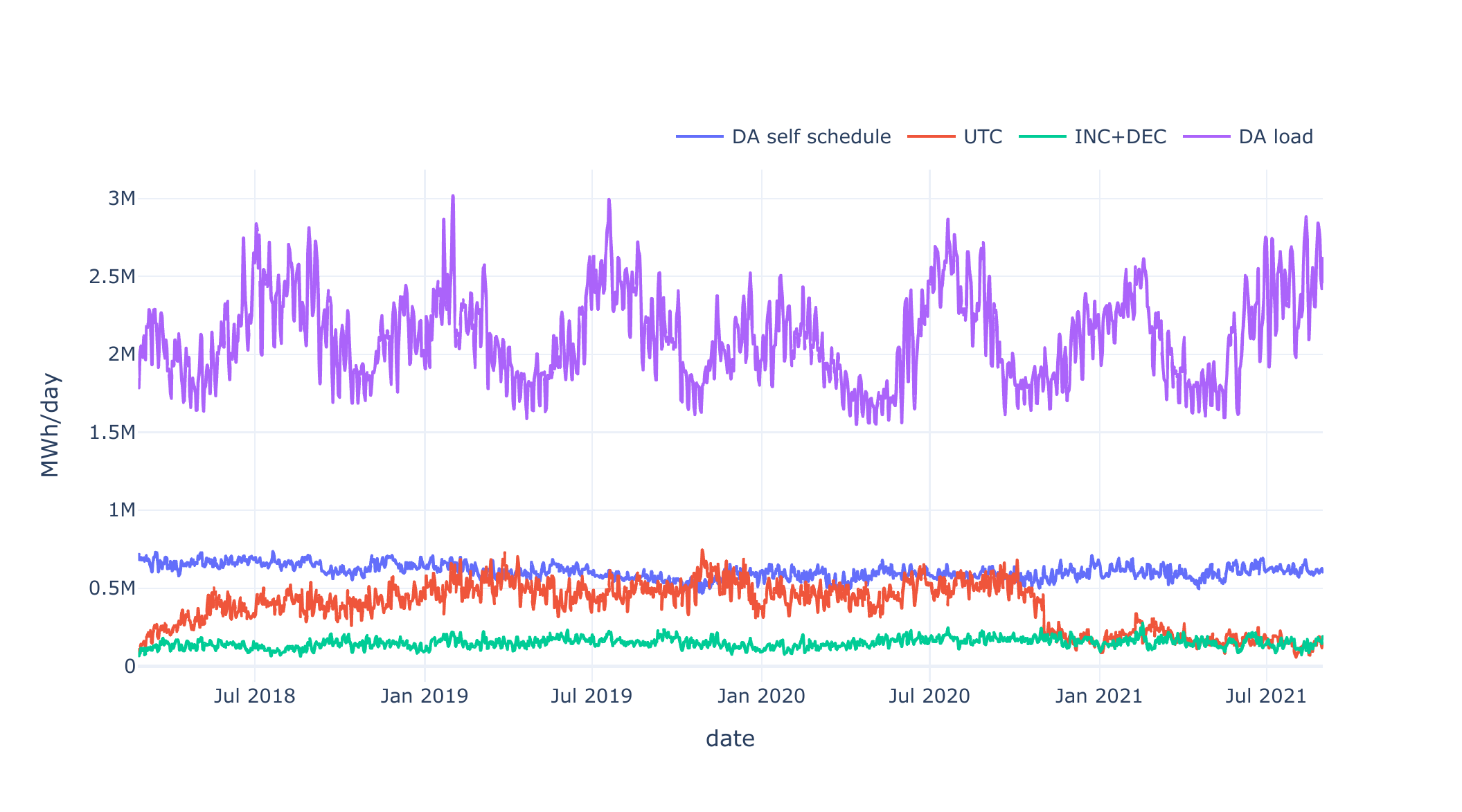}
\caption{PJM bid volume by category.}
\label{fig:PJMbids_categoriesAbs}
\end{figure}

\textsc{Self-scheduling}. Self-scheduling are the physical generation bids at zero price (as opposed to the dispatchable generation bids). PJM provides data of hourly real-time self-scheduling. We compute the volume of day-ahead self-scheduling using the bid data, as the supply offers at zero price (cf. Figure \ref{fig:PJMbids_meritOrder}). Figure \ref{fig:PJMbids_categoriesAbs} shows the resulting output, which fits well with the average figures of day-ahead self-scheduling reported by PJM market monitoring.

\textsc{Emergency procedure events}. Emergency procedure are issued by PJM to ensure reliability of the system. 
These events relate to controlled load curtailment and load shedding, generation curtailment, geomagnetic disturbance, voltage support, hot and cold weather alerts, etc.
We compute a time series of number of daily events.
We count only the events labelled as Action, Warning and Alert. If an event spans over multiple days, it is counted as an event on each of these days.
Figures \ref{fig:PJM_emergency_events} shows the resulting number of emergency events per month in PJM. 
Figure \ref{fig:PJM_emergency_procedures_message_types} shows the distribution of these events per category.

\begin{figure}
\centering
\includegraphics[width=0.8\textwidth]{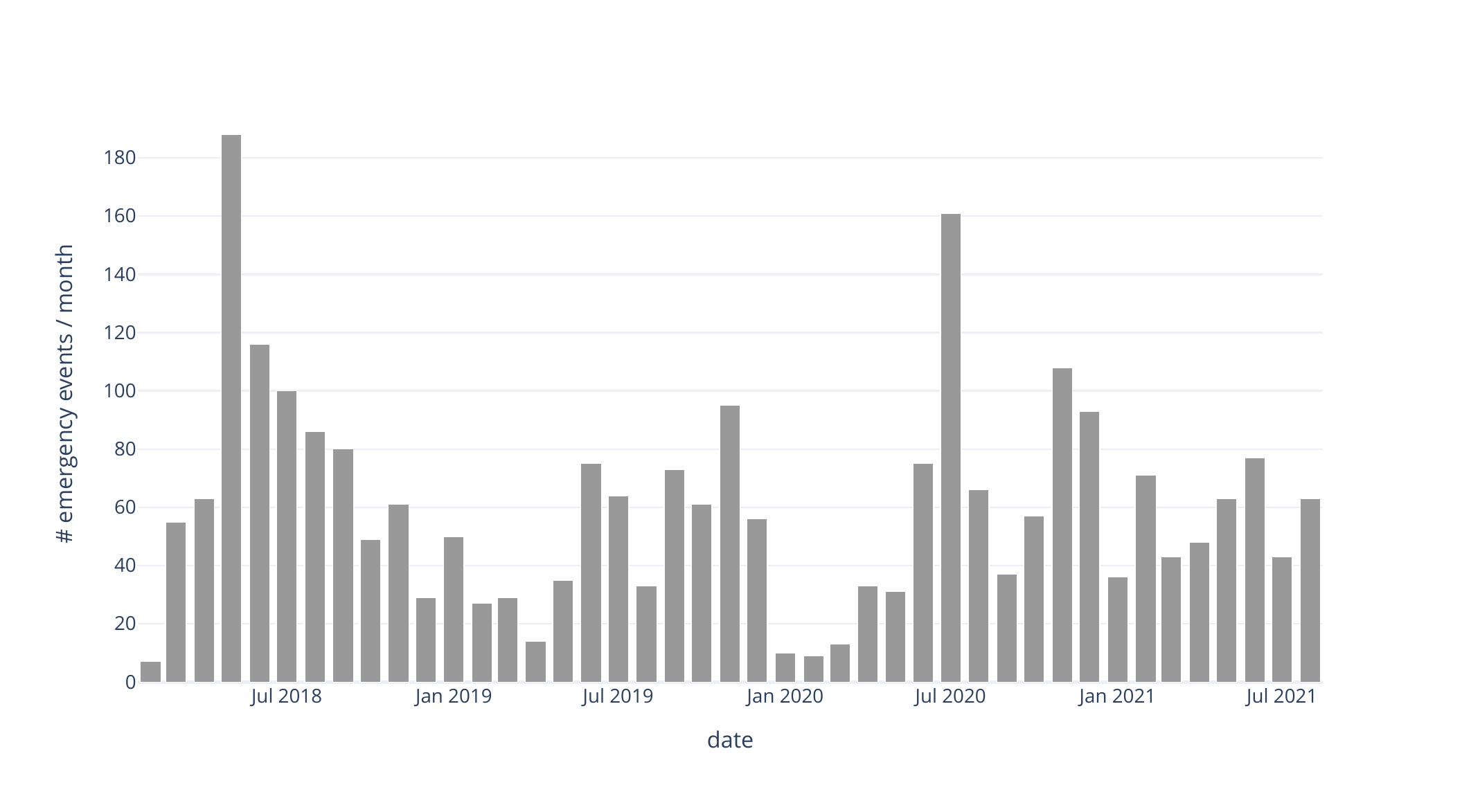}
\caption{Number of emergency events per month in PJM.}
\label{fig:PJM_emergency_events}
\end{figure}

\begin{figure}
\centering
\includegraphics[width=0.8\textwidth]{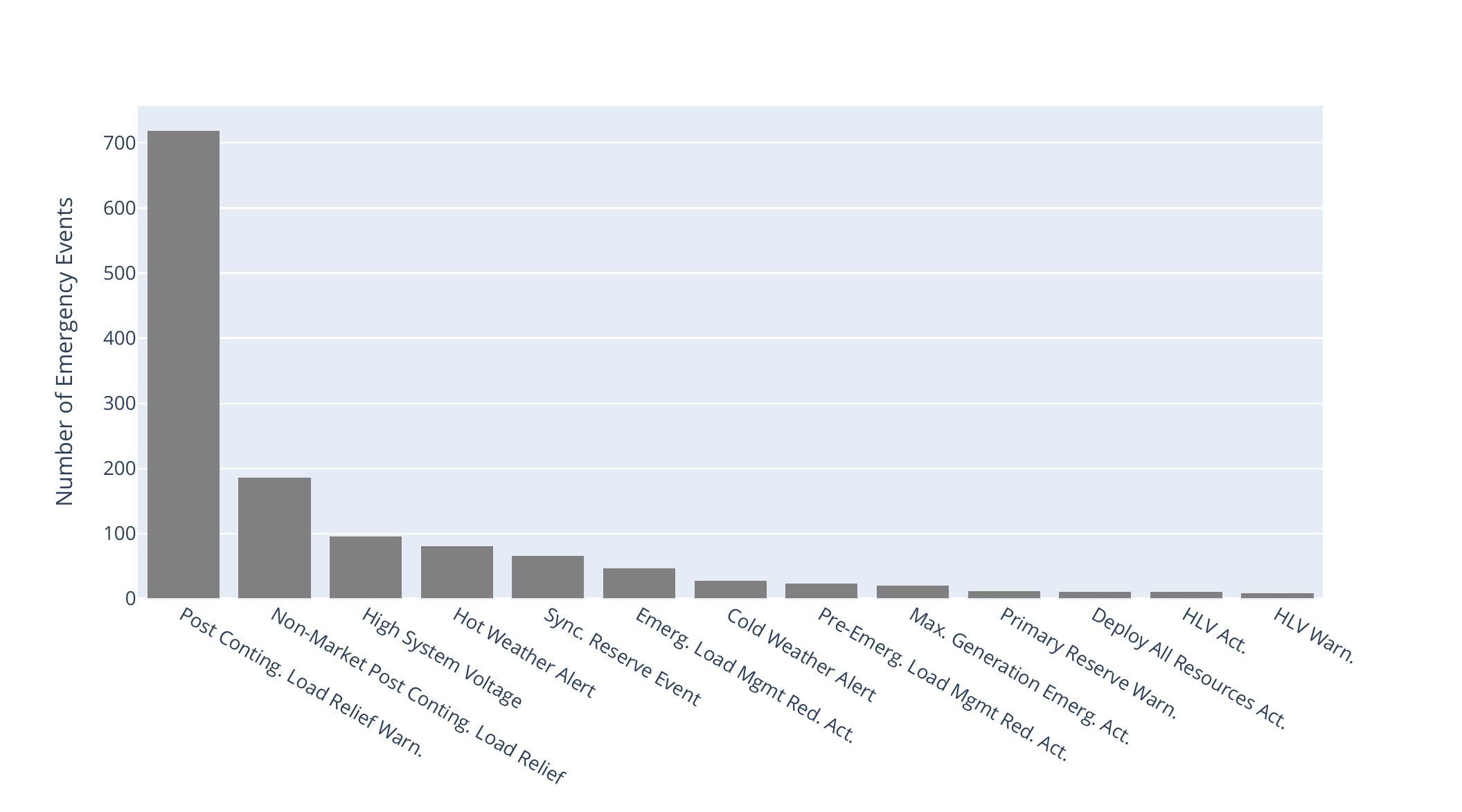}
\caption{Distribution of the emergency events per category (only the most significant categories are shown).}
\label{fig:PJM_emergency_procedures_message_types}
\end{figure}

\end{document}